\documentclass[11pt,a4paper]{article}

\usepackage[english]{babel}
\usepackage[normalem]{ulem}
\usepackage{amsfonts}
\usepackage{amsmath}
\usepackage{amssymb}
\usepackage{bbm}
\usepackage{epsfig}
\usepackage{color}
\usepackage{enumerate}
\usepackage{amsthm}
\usepackage{soul}
\usepackage{enumitem}
\newtheorem{theorem}{Theorem}
\newtheorem{lemma}[theorem]{Lemma}
\newtheorem{corollary}[theorem]{Corollary}
\newtheorem{proposition}[theorem]{Proposition}
\newtheorem{remark}[theorem]{Remark}
\newtheorem{assumption}[theorem]{Assumption}
\newtheorem{definition}[theorem]{Definition}
\newtheorem{example}{Example}
\usepackage{caption}
\usepackage{graphicx}
\usepackage{float}
\usepackage{subcaption}
\usepackage{setspace}
\usepackage{amssymb, amsfonts, amsthm,amsmath,mathrsfs}
\usepackage{ifthen,bookmark}
\usepackage{graphicx}
\usepackage{epsfig}
\usepackage{lineno,hyperref}
\usepackage{booktabs}
\usepackage{multirow}
\usepackage{bm}
\usepackage[font={bf,footnotesize},textfont=md]{caption}
\usepackage{indentfirst}
\usepackage{array}
\usepackage{arydshln}
\usepackage{graphicx}
\usepackage{listings}
\usepackage{mathtools}
\hypersetup{hidelinks}
\usepackage{float}

\usepackage[labelfont=bf]{caption}

\newcommand{\bN}{\mathbb{N}}
\newcommand{\bR}{\mathbb{R}}
\newcommand{\bC}{\mathbb{C}}
\newcommand{\bZ}{\mathbb{Z}}
\renewcommand{\Re}{\operatorname*{Re}}
\renewcommand{\Im}{\operatorname*{Im}}
\newcommand{\real}{\operatorname*{Re}}
\newcommand{\imag}{\operatorname*{Im}}
\newcommand{\tol}{\mathrm{tol}}

\newcommand{\opi}{\mathrm{i}}

\newcommand{\e}{\mathrm{e}}

\newcommand{\vv}{{\mathbf{v}}}

\newcommand{\lv}{{\mathbf{l}}}

\begin{document}

\title{Generalized Convolution Quadrature for non smooth sectorial problems}
\author{J. Guo \thanks{School of Science and Engineering, The Chinese University of Hong Kong, Shenzhen, Guangdong, 518172, P.R.
		China; And Shenzhen International Center for Industrial and Applied Mathematics, Shenzhen Research
		Institute of Big Data, Guangdong 518172, China.  Email: {\tt jingguo@cuhk.edu.cn }} \and M. Lopez-Fernandez\thanks{Department of Mathematical Analysis, Statistics and O.R., and Applied Mathematics.
Fa\-cul\-ty of Sciences. University of Malaga.
Bulevar Louis Pasteur, 31
29010  Malaga,  Spain. Email: {\tt maria.lopezf@uma.es}}}
\maketitle

\begin{abstract}
We consider the application of the generalized Convolution Quadrature (gCQ) to approximate the solution of an important class of sectorial problems. The gCQ is a generalization of  Lubich's Convolution Quadrature (CQ) that allows for variable steps. The available stability and convergence theory for the gCQ requires non realistic regularity assumptions on the data, which do not hold in many applications of interest, such as the approximation of subdiffusion equations. It is well known that for non smooth enough data the original CQ, with uniform steps, presents an order reduction close to the singularity. We generalize the analysis of the gCQ to data satisfying realistic regularity  assumptions  and provide sufficient conditions for stability and convergence on arbitrary sequences of time points. We consider the particular case of graded meshes and show how to choose them optimally, according to the behaviour of the data. An important advantage of the gCQ method is that it allows for a fast and memory reduced implementation. We describe how the fast and oblivious gCQ can be implemented and illustrate our theoretical results with several numerical experiments.
\end{abstract}

{\bf Keywords:} fractional integral, fractional differential equations, generalized convolution quadrature, variable steps, graded meshes.

{\bf AMS subject classifications:} 65R20, 65L06, 65M15,26A33,35R11.

\section{Introduction}
Given a function $f$, we consider the approximation of linear Volterra convolutions in the abstract form
\begin{equation}\label{convolution}
c(t) = \int_0^t k(t-s)f(s)\, ds,
\end{equation}
where the convolution kernel $k$ is allowed to be vector-valued, on an arbitrary time mesh
\begin{equation}\label{mesh}
\Delta:= 0 < t_1 < \dots < t_N=T.
\end{equation}
We also consider the resolution of convolution equations like \eqref{convolution}, where the data is $c$ and the goal is to approximate $f$. The efficient discretization of \eqref{convolution} and, more generally, equations involving memory terms like \eqref{convolution}, has been an active field of research for decades due to the many applications involved. Indeed the kernel can be scalar valued, such  as $k(t)=t^{\gamma}$ with $\gamma>-1$, as it is the case for the fractional integral \cite{BanLo20,JingRebecca_fi}, transparent boundary conditions \cite{Schadle+_tbc,LuScha2002}, impedance boundary conditions \cite{HipLoPa2014}, etc.,  can be a matrix or even an operator, such as $k(t)=\e^{At}$, meaning  the semigroup generated by the operator $A$, as it appears in the variation of constants formula for abstract initial value problems \cite{SchaLoLuPa}, see also \cite{CuLuPa} for other operator-valued kernels, or can even be a distribution, like the Dirac delta, in the boundary integral formulation of wave problems \cite[Chapter 2]{BanSay22}.

Lubich's Convolution Quadrature (CQ) \cite{Lu88I} is nowadays a very well established family of numerical methods for the approximation of these problems, but its construction and analysis is strongly limited to uniform time meshes, where $t_n=nh$, with $h=\frac{T}{N}$ fixed. CQ schemes never require the evaluation of the convolution kernel $k$, which, as mentioned above, is allowed to be weakly singular at the origin, be defined in a distributional sense \cite{Lu94} or even not known analytically. Instead, the Laplace transform $K$ of $k$ is used, the so called {\em transfer operator}, namely
\begin{equation}\label{genK}
K(z)=\int_0^{\infty} \e^{-zt}k(t) \,dt, \qquad \mbox{ defined for } \real z > \sigma,
\end{equation}
for a certain abscissa $\sigma \in \bR$. The application of the CQ requires a bound of the form
\begin{equation}\label{genboundK}
\|K(z)\| \le M |z|^\mu,  \qquad \mbox{ for } \real z > \sigma,
\end{equation}
for some $M>0$ and $\mu \in  \bR$, with $\mu>0$ typically holding in applications to hyperbolic time-domain integral equations  \cite{Lu94}. In this rather general setting, the generalization of Lubich's CQ to variable steps has been developed in \cite{LoSau13,LoSau15apnum,LoSau16}, presenting the so-called generalized Convolution Quadrature (gCQ), with a special focus on the resolution of hyperbolic time-domain integral equations. The available stability and convergence analysis of the gCQ requires, in general, strong regularity hypotheses on the data. In particular, for $\mu = -\alpha$, with $\alpha \in (0,1)$, convergence of the first order can be proven on a general mesh only if the data $f$ is of class $C^3([0,T],B)$ and satisfies $f(0)=f'(0)=0$, \cite[Theorem 16]{LoSau16}. Notice that the result in \cite{LoSau16} improves the first one proven in \cite{LoSau13}, where an {\em a priori} regularization step is required for problems satisfying Assumption~\ref{assumptionK}, and the data is required to satisfy $f\in C^4([0,T],B)$ with $f^{(\ell)}(0)=0$, with $0\le \ell \le 3$, see also \cite[Theorem 2.32]{BanSay22}.

There are however important applications, such as the time integration of parabolic problems \cite{LuOs}, the approximation of fractional integrals and derivatives \cite{Lu2004} and the resolution of fractional diffusion problems \cite{CuLuPa}, where $K$ can  be holomorphically extended to the complement of an acute sector around the negative real axis, and it satisfies a bound of the type
\begin{equation}\label{sectorial}
\| K(z) \| \le M |z|^{-\alpha}, \quad |\arg(z)| < \pi-\varphi,
\end{equation}
for some $\varphi \in (0,\frac{\pi}{2})$, $M = M(\varphi)>0$, and $\alpha >0$. If $K$ satisfies \eqref{sectorial}, then it is said to be a {\em sectorial Laplace transform} and we also say that $k$ is a {\em sectorial kernel}. The current theory for the generalization of the CQ to variable steps is suboptimal for these problems, where the extra regularity of the kernel should allow to relax the smoothness requirements on the data $f$ (or $c$, if we are solving the convolution equation). In the present paper we address the {\em a priori} analysis of the gCQ for this important class of problems. More precisely,  as a first step in the development of this theory, we consider following class of convolution operators.

\begin{assumption}\label{assumptionK}
Let $B$ and $D$ be some normed vector spaces and $\mathcal{L}(B,D)$ the space of linear continuous mappings from $B$ to $D$, with the usual operator norm
\[
\|F\|_{D \leftarrow B} = \sup_{u\in B, u \ne 0} \frac{\|Fu\|_D}{\|u\|_B}.
\] 
We assume that the Laplace transform of the kernel $k$ in \eqref{convolution}, is an operator valued mapping $K: \bC \to \mathcal{L}(B,D)$ which satisfies
\begin{enumerate}
\item $K$ is holomorphic in any sector $|\arg(z)|<\pi$.
\item There exist $M>0$ and $\alpha \in (0,1)$ such that
\begin{equation}\label{Kz}
\Vert	K(z)\Vert_{D \leftarrow B}\le M  |z|^{-\alpha},  \qquad |\arg(z)|<\pi.
\end{equation}
\item $K$ is continuous in the upper half plane $\imag z \ge 0$, with the possible exception of $z=0$, and similarly on the lower edge of the cut.
\end{enumerate}
\end{assumption} 
In what follows, we will omit subindexes in the norms when they are clear from the context.
Important examples of kernels satisfying Assumption~\ref{assumptionK} appear often in the li\-te\-ra\-tu\-re, such as:
\begin{enumerate}
\item The fractional integral of order $\alpha \in (0,1)$, where $k(t)=\frac{t^{\alpha-1}}{\Gamma(\alpha)}$ and $K(z)=z^{-\alpha}$.

\item The resolvent of a symmetric positive definite elliptic operator $A$ or a discrete version of it by the standard finite difference or the Finite Element method at $z^{\alpha}$, with $\alpha \in (0,1)$, provided that the spectrum of $A$ is confined to the negative real axis, away from 0. Then $K(z)=(z^{\alpha}I-A)^{-1}$ is either an operator between appropriate functional spaces or just a matrix and falls into this class of problems, since typically the resolvent of $A$ is bounded like
    \[
    \| (z I -A)^{-1} \| \le M |z|^{-1}, \quad \mbox{ for } |\arg(z)| > \pi-\varphi, \quad M=M(\varphi) >0,
    \]
    and any $\varphi \in (0,\frac{\pi}{2})$. The approximation of subdifussion equations by applying the CQ and also the gCQ to discretize the fractional derivative leads to the analysis of these methods for $k(t)=\mathcal{L}^{-1}[(z^{\alpha}I-A)^{-1}]$. This application is carefully discussed in Section~\ref{sec:fracdiffeq}.


\item Models for wave propagation in lossy media, such as the Westervelt equation studied in \cite{BakerBanjaiPtashnyk2023}, include damping terms of the form considered in this manuscript.
\end{enumerate}

Under Assumption~\ref{assumptionK}, the convolution kernel $k$ can be expressed as the inverse Laplace transform of $K$ by means of a real integral representation, see for instance \cite[Theorem 10.7d]{Henrici_II}, more precisely we can write
\begin{equation}\label{realiLT}
k(t) = \int_0^{\infty} e^{-xt} G(x) \,dx,
\end{equation}
with $G$ given by
\[
G(x) = \frac{1}{2\pi\opi} \left(K(\e^{-\opi \pi} x)-K(\e^{\opi \pi} x)\right),
\]
and bounded by
\begin{equation}\label{boundG}
\|G(x)\| \le \frac{M}{\pi} x^{-\alpha}, \qquad \mbox{ for } \ x>0.
\end{equation}
For the regularity of the data $f$, we assume that
\begin{equation}\label{dataform}
f(t)=t^\beta g(t), \ \mbox{ with $\beta >-1$,\quad $g$ sufficiently smooth}.
\end{equation}
Notice that $g$ continuous in $[0,T]$ is enough for \eqref{convolution} to be well defined under the assumptions on $k$. In this situation, the original CQ methods are well known to suffer an order reduction. Typically a CQ formula of maximal order $p$ is able to achieve that order in the $L^{\infty}$ norm only if the zero-extension of the data $f$ to $t<0$ is smooth enough, depending on $\alpha$ and $p$ . This question has been thoroughly analyzed in the literature, starting from the introduction of the CQ itself in \cite{Lu88I}. More precisely, we follow the operational notation introduced by Lubich and set
\begin{align}
K(\partial_t)f & = k*f, \qquad \mbox{the continuous convolution \eqref{convolution}}, \\
K(\partial^{h}_t)f, &  \qquad \mbox{the approximation of \eqref{convolution} by the original CQ with step $h$},  \\
K(\partial^{\Delta}_t)f, &  \qquad \mbox{the approximation of \eqref{convolution} by the gCQ on the time mesh $\Delta$ in \eqref{mesh}}.
\end{align}\label{opnotation}
We consider the particular case $f(t)=t^{\beta}\vv$, with a $t$-independent $\vv \in B$, cf. \cite{CuLuPa}.
Then, for $p=1$, the following error estimates follow from \cite[Theorem 5.2, Corollary 3.2]{Lu88I}
\begin{equation}\label{errorCQ}
\left\|\left[K(\partial_t)f\right](t_n)-\left[K(\partial^{h}_t)f\right]_n \right\| \le
\left\{
\begin{array}{ll}
C t_n^{\alpha-1} h^{\beta+1}, \ & \mbox{ for } -1 < \beta \le  0,  \\
C t_n^{\alpha + \beta -1} h, \  & \mbox{ for }  \beta \ge  0.
\end{array}
\right.
\end{equation}
The above estimates imply that the order of convergence close to the origin is actually $\alpha+\beta$, and that the maximal order of convergence (which is one) is only achievable pointwise at times away from the origin and for $\beta\ge 0$. Similar error estimates in time hold for the application of the CQ to linear subdiffusion equations \cite{jin2016two}, whose solutions are known to satisfy \eqref{dataform}, see \cite[Section 8]{CuLuPa}.
For approximations of order higher than one, correction terms are needed in order to achieve the full maximal order at a given time point away from zero but, in any case, the convergence rate still deteriorates close to the singularity \cite{CuLuPa,JinLiZhou2017}.

The real integral representation in \eqref{realiLT} of the convolution kernel is much simpler to deal with than the usual contour integral representation along a Hankel contour in the complex plane, which is used in the analysis of more general sectorial problems \cite{LuOs}. This allows us to derive rather clean error estimates for the gCQ method and still brings quite a lot of insight for the analysis of more general sectorial problems and also of higher order gCQ schemes, which we do not address in the present paper. On the other hand, the definition of the gCQ based on \eqref{realiLT} is equivalent to the original one in \cite{LoSau13} and thus all the important properties proven in \cite{LoSau13} still hold, such as the preservation of the composition rule, see Remark~\ref{remark:equivalence_gCQ}. Being able to use  \eqref{realiLT} also brings  important advantages from the implementation point of view. In particular, we refer to the fast and oblivious CQ algorithm for the fractional integral and associated fractional differential equations developed in \cite{BanLo19}. In this paper,  we generalize the algorithm in \cite{BanLo19} to the gCQ method, and show how both the computational complexity and the memory can be drastically reduced even for variable step approximations of the convolution problems under study.


The real integral representation of the kernel $t^{\alpha-1}/\Gamma(\alpha)$ has been recently used in \cite{BanMak2022} to derive {\em a posteriori} error formulas for both the L1 scheme and the gCQ of the first order. For the popular L1 method the maximal order of convergence is known to be $2-\alpha$, higher than for the gCQ of the first order. In \cite{BanMak2022}, the maximal order $2-\alpha$ is shown to be achievable in the $L^2$ norm by using appropriate graded meshes. A posteriori error formulas are also derived  for the gCQ method of the first order, but their asymptotic analysis is not developed.  The authors point indeed to the complicated a priori error analysis which is required for this. In this paper we address precisely this a priori error analysis and derive error bounds in the $L^{\infty}$ norm for the gCQ method of the first order. We obtain optimal error estimates of the maximal order for general meshes and, as a particular case, for appropriate graded meshes. By doing this, we fill an important gap in the theory in \cite{LoSau16}, for an important family of applications.

As we have already mentioned earlier, our analysis is conducted on an arbitrary time grid, but  we consider in detail the particular choice of graded meshes, defined by
\begin{equation}\label{gmesh}
t_n= (n\tau)^{\gamma}, \qquad \tau=\frac{T^{1/\gamma}}{N},\quad n=1,\dots,N.
\end{equation}
For \eqref{gmesh} we derive the optimal value of the grading parameter $\gamma$ as a function of $\alpha$ and $\beta$ in order to achieve full order of convergence. 

The paper is organized as follows. In Section~\ref{sec:gcq} we recall the definition of the generalized Convolution Quadrature method. In Section~\ref{sec:analysis} we analyze this method for $f$ satisfying \eqref{dataform} and for sectorial kernels satisfying \eqref{realiLT}-\eqref{boundG}. In Section~\ref{sec:fracdiffeq} we address the application to linear subdiffusion equations. Sections~\ref{sec:algfi} and \ref{sec:algsubdiff} are devoted to implementation issues and numerical results are shown in Section~\ref{sec:experiments}.

\section{Generalized Convolution Quadrature of the first order}\label{sec:gcq}
We present the gCQ method for \eqref{convolution} under the assumptions \eqref{realiLT}-\eqref{boundG} in a different way from the derivation in \cite{LoSau13} and \cite{LoSau16}, which takes into account the specific properties of the convolution kernels under study. Indeed \eqref{realiLT} allows to write
\begin{equation}\label{convo_realint}
c(t) = \int_{0}^{\infty} G(x) \int_0^t \e^{-x(t-s)} f(s)\,ds \, dx
 =   \int_{0}^{\infty} G(x) y(x,t) \, dx,
\end{equation}
with $y(x,t)$ the solution to the scalar ODE problem
\begin{equation}\label{ode}
\partial_t y(x,t) = -x y(x,t) + f(t), \quad \mbox{ with }  \quad y(x,0)=0, \quad  x\in(0,\infty),\quad t\in(0,T].
\end{equation}
With this representation of the convolution kernel, the gCQ approximation of \eqref{convolution} is defined below.
\begin{definition}
For a given $N \ge 0$ and a sequence of time points
\begin{equation}\label{gentimes}
\Delta:= 0<t_1<\dots < t_N = T, \qquad \mbox{ with }\quad \tau_j := t_j-t_{j-1}, \quad j \ge 1,
\end{equation}
the generalized Convolution Quadrature approximation to \eqref{convolution} based on the implicit Euler method is given by
\begin{equation}\label{gCQ}
c_n  = \int_{0}^{\infty} G(x) y_n(x) \, dx,
\end{equation}
where $y_n(x)$ is the approximation of $y(x, t_n)$ in \eqref{ode} given by the implicit Euler method, this is,
\begin{equation}\label{eulersol}
y_n(x)= \sum_{j=1}^{n} \tau_j  \left(\prod_{\ell=j}^{n} \frac{1}{1+\tau_{\ell} x}\right) f(t_j).
\end{equation}
\end{definition}
In applications it will be convenient to use the operational notation \eqref{operationalconvo}, which is defined by
\begin{equation}\label{gCQshort}
\left[ K(\partial_t^{\Delta})f \right]_n = c_n  = \sum_{j=1}^{n} \omega_{n,j} f_j,
\end{equation}
with $f_j=f(t_j)$, if $f$ is a function. We will use the same notation for $f$ being a vector, and in this case, $f_j$ denotes its $j$-th component. The {\em gCQ weights} are given by
\begin{equation}\label{gcqw}
\omega_{n,j}:= \tau_j  \int_{0}^{\infty} G(x) \prod_{\ell=j}^{n} r(\tau_\ell x)\, dx, \qquad r(x):=\frac{1}{1+x}.
\end{equation}
Notice that $r(x)=R(-x)$, with $R(z)=\frac{1}{1-z}$ the stability function of the implicit Euler method.

\begin{remark}\label{remark:equivalence_gCQ}
Under assumptions  \eqref{realiLT}-\eqref{boundG}, our definition of the gCQ approximation \eqref{gCQ} to \eqref{convolution} is equivalent to the one in \cite{LoSau13}, where the gCQ weights are shown to be given by
\begin{equation}\label{originalgCQ}
\omega_{n,j}:= \left(\prod_{\ell=j+1}^{n} \left(-\tau_{\ell}\right)^{-1}  \right) K \left[\tau_j^{-1},\dots, \tau_n^{-1} \right],
\end{equation}
with $K \left[\tau_j^{-1},\dots, \tau_n^{-1} \right]$ Newton's divided differences. Thus the properties proven in \cite{LoSau13} for this scheme hold also under the equivalent definition \eqref{gcqw}. In particular, the gCQ method has the very interesting property of preserving the composition rule. More precisely, by using Lubich's operational notation for one-sided convolutions \cite{Lu94}
\begin{equation}\label{operationalconvo}
\left(K(\partial_t)f\right) (t):= \int_{0}^{t} k(t-s) f(s) \, ds, \qquad \mbox{ with }\quad K=\mathcal{L}[k],
\end{equation}
it holds
\begin{equation}\label{compoconvo}
K_2(\partial_t)K_1(\partial_t)f = (K_2 \cdot K_1)(\partial_t)f.
\end{equation}
Following the notation in \cite{LoSau16}, if we denote by $K(\partial^{\Delta}_t)f$ the vector of approximations to $\left(K(\partial_t)f\right) (t_j)$ on a general time grid
\begin{equation}\label{grid}
\Delta := 0 < t_1< \dots < t_N=T,
\end{equation}
by the gCQ method, it is proven in \cite[Section 5.2]{LoSau16} that
\begin{equation}\label{compogCQ}
K_2(\partial^{\Delta}_t)K_1(\partial^{\Delta}_t)f = (K_2 \cdot K_1)(\partial^{\Delta}_t)f.
\end{equation}
This property is essential in many error proofs of CQ based schemes for different problems involving memory terms of convolution type and will be used in Section 4 when discussing the application of the gCQ to subdiffussion equations.

In  \cite{LoSau13} a complex contour integral representation of the convolution kernel is used. Here we use the real integral representation of the kernel \eqref{realiLT}, which can actually be derived from deformation of the complex contour, as explained in \cite{BanLo19}. This representation of the gCQ based on the real inversion of the Laplace transform is both useful for the error analysis in the low regularity setting, as we show in Section~\ref{sec:analysis}, and for the implementation of the method, see \cite{BanLo19}, and Sections~\ref{sec:algfi} and \ref{sec:algsubdiff}.
\end{remark}

We recall here the best convergence result available so far for the gCQ based on the implicit Euler method, which is a particular case of the class of Runge--Kutta based gCQ studied in \cite{LoSau16}.

\begin{theorem}[{\cite[Theorem 16]{LoSau16}}]\label{th:convergence_regular} Let $0<\sigma<\tau_{\max}^{-1}$ with
\begin{equation}\label{taumax}
\tau_{\max} := \max_{1\le j\le N} \{ \tau_j \}.
\end{equation}
Assume that $K$  satisfies \eqref{genboundK} with $\mu = -\alpha$, for some $\alpha \in (0,1)$. Then if
$f\in C^{3}\left( 0,T \right)$, with $f^{\left(  \ell\right)  }\left(  0\right)  =0$, for $\ell = 0,1$, there exist constants $C, \tilde c>0$ such that the following error estimate holds
\[
\Vert c(t_n)-c_n \Vert_D \le C \e^{\tilde c\sigma T}
		\left\Vert f\right\Vert_{C^{3}([0,T],B)} \tau_{\max}.
\]
\end{theorem}%

In the next section we show how the regularity requirements for $f$ can be significantly relaxed under assumptions \eqref{realiLT}-\eqref{boundG}. Furthermore, we will also be able to remove the exponentially growing term $\e^{\tilde c\sigma T}$ in the error estimate.

\section{Error analysis in the non smooth case}\label{sec:analysis}

From \eqref{convo_realint} and \eqref{gCQ} we can write the error at time $t_n$ by
\begin{equation}\label{En}
c(t_n)-c_n = \int_{0}^{\infty} G(x) e_n(x) \, dx,
\end{equation}
with
\[e_n(x)=y(x,t_n)-y_n(x).\]
As it is usually done to analyze the error for ODE solvers, we notice that $y(x,t)$ solves the recurrence
\[
\frac{y(x,t_n)-y(x,t_{n-1})}{\tau_n} = -x y(x,t_n)+f(t_n)+d_n(x),
\]
with $d_n$ Euler's remainder, this is,
\begin{equation}\label{remainder}
d_n(x) = \frac{y(x,t_n)-y(x,t_{n-1})}{\tau_n} - y_t(x,t_n), \qquad n\ge 1.
\end{equation}
Then
\begin{equation}\label{en}
e_n(x)=\sum_{j=1}^{n} \tau_j d_j(x) \prod_{\ell=j}^{n} r(\tau_\ell x).
\end{equation}

For the special case of $f(t)= t^{\beta}\vv$, with  $\beta >-1$, and $t$-independent \( \vv \in B\), we can obtain explicit representations of \eqref{en}. For this, we recall the definition of the Mittag-Leffler function, see for instance \cite{Gorenflo2020}.

\begin{definition}[Mittag-Leffler function] For $\alpha>0$, $\beta\in\mathbb{R}$, the two parameter Mittag-Leffler function $E_{\alpha,\beta}(z)$ is defined by
$$E_{\alpha,\beta}(z)=\sum\limits_{l=0}^{\infty}\frac{z^l}{\Gamma(\alpha l+\beta)},\quad z\in\mathbb{C}.$$
\end{definition}

We also recall the definition of the Riemann-Liouville fractional derivative of order $\mu >0$, which is given by
\begin{equation*}
	f^{(\mu)}(t)=\frac{d^{\rho}}{dt^{\rho}} \mathcal{I}^{\rho-\mu}[f](t),
\end{equation*}
with  $\mu>0$, $\rho-1\le\mu<\rho$, $\rho\in \mathbb{N}$, and $\mathcal{I}^{\rho-\mu}$ the fractional integral of order $\rho-\mu$, with
\begin{equation}\label{fracint}
\mathcal{I}^{\nu}[f](t)=\frac{1}{\Gamma(\nu)}\int_0^t (t-s)^{\nu-1} f(s)\,ds,
\end{equation}
if $\nu>0$. For $\mu < 0$ we take $f^{(\mu)}=\mathcal{I}^{-\mu}$. Notice that
\[
\mathcal{I}^{\nu}[f] = \partial_t^{-\nu }f,
\]
with the operational notation \eqref{operationalconvo}.

In the next result we collect a few identities and properties that will be needed in the proof of Theorem~\ref{thm:errgCQ}.
\begin{lemma}\label{lema:identities}
It holds:
\begin{enumerate}
\item  If $f(t)=t^{\ell}$, with $\ell>-1$,
\begin{equation}\label{fracint_tell}
\mathcal{I}^{\alpha}[f](t) = \frac{\Gamma(\ell+1)}{\Gamma(\alpha + \ell+1)}t^{\alpha+\ell}.
\end{equation}
\item For $\beta > -1$, $\beta \ne 0$, the solution $y(x,t)$ to \eqref{ode} with $f(t)=t^{\beta} \vv$ satisfies
\begin{eqnarray}
y(x,t) &=&  \Gamma(1+\beta) t^{\beta+1}E_{1,\beta+2}(-xt) \vv,\label{solode_ml} \\[.5em]
y_t(x,t) &=&  \Gamma(1+\beta) t^{\beta}E_{1,\beta+1}(-xt)  \vv, \label{soldode_ml}\\[.5em]
y_{tt}(x,t) &=& \Gamma(1+\beta) t^{\beta-1} E_{1,\beta}(-xt) \vv.\label{solddode_ml}
\end{eqnarray}
For $\beta=0$ \eqref{solode_ml} and \eqref{soldode_ml} still hold true, but \eqref{solddode_ml} must be replaced by
\begin{equation}\label{solddode_exp}
y_{tt}(x,t) = -x\e^{-xt}\vv.
\end{equation}

\item From \cite[Theorem 1.6]{Podlubny}, for $\alpha<2$ and $\beta \in \bR$, there exists $C>0$ such that
\begin{equation}\label{bound-absml-neg}
|E_{\alpha,\beta}(-y)| \le \frac{C}{1+y}, \qquad y\ge 0.
\end{equation}
\end{enumerate}
\end{lemma}
\begin{proof}
Identity \eqref{fracint_tell} follows directly by noticing that the Laplace transform of the left hand side is $\Gamma(\ell+1) z^{-\ell-\alpha-1}$.

In order to prove \eqref{solode_ml}, we write
\[
y(x,t) =\int_{0}^{t} (t-s)^{\beta} \e^{-xs}\, ds \vv= \Gamma(1+\beta)\mathcal{I}^{1+\beta}[\e^{-x \cdot}](t)\vv.
\]
By expanding the exponential into its power series and using the same argument as to prove \eqref{fracint_tell}, we can write
\begin{align*}
y(x,t) &=\sum_{\ell=0}^{\infty} \frac{(-x)^{\ell}}{\ell!} \int_{0}^{t} (t-s)^{\beta} s^{\ell}\, ds \vv\\
&= \Gamma(1+\beta) \sum_{\ell=0}^{\infty} \frac{(-x)^{\ell}}{\ell!} \frac{\Gamma(\ell+1)}{\Gamma( \ell+\beta +2)}t^{\beta+\ell+1}\vv\\
&=\Gamma(1+\beta) t^{\beta+1}E_{1,\beta+2}(-xt)\vv,
\end{align*}
this is \eqref{solode_ml}.

Derivation of the power series gives
\[
y_t(x,t) =\Gamma(1+\beta) \sum_{\ell=0}^{\infty} \frac{(-x)^{\ell}} {\Gamma( \ell+\beta +2)}(\beta+\ell+1)t^{\beta+\ell} \vv  =\Gamma(1+\beta) \sum_{\ell=0}^{\infty} \frac{(-x)^{\ell}} {\Gamma( \ell+\beta +1 )}t^{\beta+\ell}\vv
\]
and so \eqref{soldode_ml}. The expression \eqref{solddode_ml} for $\beta \ne 0$ follows  analogously. For $\beta=0$ the term in $\ell=0$ is constant and thus we must use instead \eqref{solddode_exp}, which follows by differentiating twice in
\begin{equation}\label{solode_exp}
y(x,t) = \frac{1-\e^{-xt}}{x}\vv.
\end{equation}
\end{proof}

\begin{lemma}\label{lema:dkml}
Let $\beta >-1$, and $f(t)=t^{\beta}\vv$ in \eqref{ode}. Then
\begin{equation}\label{d1}
d_1(x) = -\tau_1^{\beta} \Gamma(1+\beta) \left(E_{1,\beta+1}(-x\tau_1)-E_{1,\beta+2}(-x\tau_1)  \right)\vv
\end{equation}
and, for $j>1$, there exist $\xi_j \in(t_{j-1},t_j)$ such that
\begin{equation}\label{dk-ml}
d_j(x) = \left\{ \begin{array}{ll}
\displaystyle
-\frac{\tau_j}{2} \Gamma(1+\beta) \xi_j^{\beta-1} E_{1,\beta}(-x\xi_j)\vv,& \mbox{ if }\beta \ne 0, \\[1em]
\displaystyle
-\frac{x\tau_j}{2}\e^{-x\xi_j}\vv , & \mbox{ if }   \beta = 0.
\end{array}
  \right.
\end{equation}

\end{lemma}

\begin{proof} For $j=1$ the result follows from \eqref{solode_ml} and \eqref{soldode_ml}, since
\[
d_1(x) = \frac{1}{\tau_1} \left( y(x,\tau_1) - y_t(x,t_1) \right). 
\]
For $j\ge 2$, we apply \eqref{solddode_ml} or \eqref{solddode_exp}, according to the value of $\beta$, to
\[
d_j(x) =\frac{1}{\tau_j}\left( y(x,t_j)-y(x,t_{j-1}) \right)- y_t(x,t_j) = -\frac{\tau_j}{2} y_{tt}(x,\xi_j).
\]

\end{proof}

The main result is given below.
\begin{proposition}\label{prop:errgCQ_powt}
 Assume that $K$ satisfies Assumption~\ref{assumptionK}. For $f(t)=t^{\beta} \vv$ with $\beta>-1$,   $\vv\in B$ independent of $t$, and for an arbitrary sequence of time points $0<t_1<\dots<t_n$  with step sizes $\tau_j=t_j-t_{j-1}$, $j\ge1$,  there exist $\xi_j \in (t_{j-1},t_j)$, for $j\ge 2$, such that the error in \eqref{En} can be bounded by
\begin{equation}\label{bgen_en}
\Vert c(t_n)-c_n \Vert_D\le C \Vert \vv\Vert_B \left(\tau_1^{\alpha+\beta}+\sum_{j=2}^{n} \tau_j^{2} \xi^{\alpha+\beta-2}_j\right),
\end{equation}
with  $C$ depending on  $\alpha$, $\beta$ and the constant in \eqref{boundG}.

In particular, for $t_n$ as in \eqref{gmesh}, $1\le n \le N$, it holds
\begin{equation}
 \Vert c(t_n)-c_n \Vert_D \leq C \Vert \vv\Vert_B T^{\alpha+\beta} \left\{
\begin{array}{ll}
N^{-\gamma(\alpha+\beta)} , & \quad \gamma(\alpha+\beta)<1, \\
 N^{-1}(1+\log(n)) , & \quad \gamma(\alpha+\beta)=1, \\
 N^{-1} t_n^{\alpha+\beta-1/\gamma}, & \quad \gamma(\alpha+\beta)>1.
\end{array}
\right.
\end{equation}
\end{proposition}

\begin{proof}
We write
\begin{align*}
c(t_n)-c_n &= \int_0^{\infty} G(x) e_n(x) \, dx = \sum_{j=1}^{n} \tau_j   \int_0^{\infty} G(x)  \left(\prod_{\ell=j}^{n} r(\tau_\ell x) \right) d_j(x) \,dx.
\end{align*}
By \eqref{boundG} and noticing that $0< r(\tau_\ell x) \le 1$, we can bound
\[
\Vert  c(t_n)-c_n\Vert_D \le  \frac{M}{\pi} \sum_{j=1}^{n} \tau_j  \int_0^{\infty} x^{-\alpha}   \Vert d_j(x) \Vert_B  \,dx,
\]
provided that the integral above is convergent.

For $j=1$ we use \eqref{d1} and \eqref{bound-absml-neg} to obtain
\begin{align*}
&\tau_1   \int_0^{\infty} x^{-\alpha}  \Vert d_j(x)  \Vert_B   \,dx \\ & \le \Vert \vv\Vert_B \tau_1^{\beta+1}\Gamma(1+\beta) \int_0^{\infty} x^{-\alpha}\left( \left| (E_{1,\beta+1}(-x\tau_1) \right| +\left| E_{1,\beta+2}(-x\tau_1)\right| \right)  \,dx \\
 & \le 2C\Vert \vv\Vert_B  \tau_1^{\beta+1}\Gamma(1+\beta) \int_0^{\infty} x^{-\alpha}\frac{1}{1+x\tau_1}\,dx\\
 &=2C\Vert \vv\Vert_B  \tau_1^{\alpha + \beta}\Gamma(1+\beta) \int_0^{\infty} u^{-\alpha}\frac{1}{1+u}\,du\\
 &=2C\Vert \vv\Vert_B  \tau_1^{\alpha + \beta}\Gamma(1+\beta) \mathrm{B}(1-\alpha,\alpha),
\end{align*}
with $\mathrm{B}$ the Beta function.

For $j>1$ and $\beta = 0$ we have, by Lemma~\ref{lema:dkml},
\begin{align*}
\int_{0}^{\infty}x^ {-\alpha }  \Vert d_ j( x)  \Vert_B  \, dx & =\Vert \vv\Vert_B  \frac{\tau_j}{2} \int_{0}^{\infty}x^ {1-\alpha} \e^{-x\xi_j}\,dx\\
&= \Vert \vv\Vert_B \frac{\tau_j}{2} \xi_j^{\alpha-2} \int_{0}^{\infty}u^ {1-\alpha} \e^{-u}\,du\\
&\le \Vert \vv\Vert_B \Gamma(2-\alpha)\frac{\tau_j}{2} t_{j-1}^{\alpha-2},
\end{align*}
for $\xi_j\in(t_{j-1},t_j)$.

For $j>1$ and $\beta\ne 0$, we use \eqref{dk-ml} and \eqref{bound-absml-neg} to bound
\begin{align*}
\int_{0}^{\infty}x^ {-\alpha }   \Vert d_ j( x)  \Vert_B  \, dx &=\Vert \vv\Vert_B 
\frac{\Gamma(1+\beta)}{2} \tau_j  \xi_j^{\beta-1} \int_{0}^{\infty}x^ {-\alpha }\left| E_{1,\beta}(-x\xi_j)\right| \,dx\\
&\le
C\Vert \vv\Vert_B \frac{\Gamma(1+\beta)}{2} \tau_j  \xi_j^{\beta-1} \int_{0}^{\infty}x^ {-\alpha }\frac{1}{1+x\xi_j} \,dx\\
&=
C\Vert \vv\Vert_B \frac{\Gamma(1+\beta)}{2} \tau_j  \xi_j^{\alpha+\beta-2} \int_{0}^{\infty}u^ {-\alpha }\frac{1}{1+u} \,du, \quad  \xi_j\in(t_{j-1},t_j).
\end{align*}
This proves the general estimate \eqref{bgen_en}.

We now notice that for the graded mesh in \eqref{gmesh} we have
\begin{equation}\label{propgmesh}
\tau_j  \le \gamma \tau^{\gamma} j^{\gamma-1}, \quad j\ge 1,
\end{equation}
and that for $\beta < 2-\alpha$ we can bound $\xi_j^{\alpha+\beta-2} \le t_{j-1}^{\alpha+\beta-2}$, for $j\ge 2$. Thus
\begin{align}
\nonumber
&	\sum\limits_{j=2}^{n}\tau_j\int_{0}^{\infty}x^ {-\alpha }  \Vert d_ j( x) \Vert_B   dx \\
&\le C\Vert \vv\Vert_B   \gamma^2 \tau^{\gamma(\alpha+\beta)}  \sum\limits_{j=2}^{n} j^{2(\gamma-1)}  (j-1)^{\gamma(\alpha+\beta-2)}.
\end{align}
We bound
\[
j^{2(\gamma-1)}  (j-1)^{-2\gamma} =\left( 1+\frac{1}{j-1} \right)^{2\gamma} j^{-2} \le 4^{\gamma}j^{-2}
\]
and then
\begin{align}
\nonumber
&	\sum\limits_{j=2}^{n}\tau_j\int_{0}^{\infty}x^ {-\alpha }   \Vert d_ j( x)  \Vert_B  dx \le C\Vert \vv\Vert_B  \gamma^2 4^{\gamma} \tau^{\gamma(\alpha+\beta)}  \sum\limits_{j=2}^{n}  j^{\gamma(\alpha+\beta)-2},
\end{align}
and the result follows. For $\beta \ge 2-\alpha$ we use the bound $\xi_j^{\alpha+\beta-2} \le t_{j}^{\alpha+\beta-2}$ and derive the same result without the constant $4^{\gamma}$.
\end{proof}
\begin{remark}
	Note that for $\alpha+\beta>1$, the error estimate \eqref{bgen_en} on the uniform mesh with step size $\tau$ can be written as
\begin{align*}
 \Vert c(t_n)-c_n\Vert_D &\le C\Vert \vv\Vert_B \left(\tau^{\alpha+\beta}+\sum_{j=2}^{n} \tau^{2} \xi^{\alpha+\beta-2}_j\right)\\[.5em]
&\le C\Vert \vv\Vert_B \left(\tau^{\alpha+\beta}+\sum_{j=2}^{n} \tau^{2} \max\left(t^{\alpha+\beta-2}_j,t^{\alpha+\beta-2}_{j-1}\right)\right)\\[.5em]
&\le \frac{C\Vert \vv\Vert_B }{\alpha+\beta-1}\tau t_n^{\alpha+\beta-1},
\end{align*}
which is in agreement with the result in \cite[Theorem	5.2]{Lu88I}. Moreover, for  $\alpha+\beta=1$,  the error estimate  \eqref{bgen_en} on the uniform mesh can be bounded by
\begin{align*}
	 \Vert c(t_n)-c_n\Vert_D &\le C\Vert \vv\Vert_B \tau\left(1+\log(n)\right),\quad n\ge 1.
\end{align*}
\end{remark}
We will need the following generalization of the previous result to deal with remainder terms in power series expansions of the right hand side.
\begin{corollary}\label{coro:peano}
 Assume that $K$ satisfies Assumption~\ref{assumptionK}. Then for fixed $\sigma\ge 0$ and $f(t)=(t-\sigma)_+^{\beta} \vv$, where $\beta > 1$,   $(y)_+:=\max(0,y)$,  $\vv\in B$ is $t$-independent,  and for an arbitrary sequence of time points $0<t_1<\dots<t_n$ with step sizes $\tau_n=t_n-t_{n-1}$, $n\ge1$,  there are $\xi_j \in (t_{j-1},t_j)$, for $j\ge 2$, such that the error in \eqref{En} can be bounded by
\begin{equation}\label{bgen_en_peano}
 \Vert c(t_n)-c_n\Vert_D \leq C\Vert \vv\Vert_B  \left(\tau_1(\tau_1-\sigma)_+^{\alpha+\beta-1}	+\sum\limits_{j=2}^n\tau_j^2 (\xi_j-\sigma)_+^{\alpha+\beta-2}\right),
\end{equation}
with  $C$ depending on $\alpha$, $\beta$ and the constant in \eqref{boundG}.

\end{corollary}
\begin{proof}
We observe that
\[
y(x,t) = \left\{\begin{array}{ll}
0, & \quad  \mbox{ if }\ t < \sigma, \\
u(x,t-\sigma), & \quad  \mbox{ if }\  t \ge \sigma,
\end{array}
\right.
\]
where $u$ is the solution to \eqref{ode} with right hand side $t^{\beta}$. The restriction to $\beta  > 1$ guarantees that $y(x,t)$ in \eqref{ode} is twice differentiable with respect to $t$ in $(0,T]$, with $y_{tt}$ bounded, and obvious formulas follow for $y_t$ and $y_{tt}$. This leads to the following generalization of \eqref{d1}
\begin{align*}\label{peanod1}
	&d_1(x)=\frac{y(x,\tau_1)}{\tau_1}-y_t(x,\tau_1)\\
	&=\Gamma(1+\beta)(\tau_1-\sigma)_+^{\beta}\left(\frac{(\tau_1-\sigma)_+}{\tau_1}E_{1,\beta+2}(-x(\tau_1-\sigma))-E_{1,\beta+1}(-x(\tau_1-\sigma))\right) \vv
\end{align*}
and of \eqref{dk-ml}
\begin{equation*}\label{peanodk}
d_j(x) = -\frac{\tau_j }{2} \Gamma(1+\beta) \left(\xi_j - \sigma\right)_+^{\beta-1} E_{1,\beta}(-x(\xi_j-\sigma)) \vv,\quad 2\le j\le n.
\end{equation*}
The integrals in the proof of Proposition~\ref{prop:errgCQ_powt} are then either $0$ or can be estimated by the analogous expressions with $\tau_1$ or $\xi_j$ shifted by $\sigma$.

\end{proof}

\begin{lemma}\label{lem:fracderi}
\cite[(2.108)]{Podlubny}
Let $\mu>0$, $\rho-1\le\mu<\rho$, $\rho\in \mathbb{N}$,  if the fractional derivative $f^{(\mu)}(t)$ is integrable, then
\begin{equation*}
	\mathcal{I}^{\mu}[f^{(\mu)}](t)=f(t)-\sum_{\ell=1}^\rho\frac{f^{(\mu-\ell)}(0)}{\Gamma(\mu-\ell+1)} t^{\mu-\ell}.
\end{equation*}
\end{lemma}
According to the definition of the fractional integral, Lemma \ref{lem:fracderi} gives the fractional Taylor expansion of $f(t)$ at $t=0$ with remainder in integral form, cf.~\cite[(3.20)]{Lu86}.

\begin{remark}
If a function $f$ satisfies \eqref{dataform} with $g \in C^{\lceil\mu\rceil +1}([0,T],B)$, then $f^{(\mu)}$ is conti\-nuous in $[0,T]$, see Appendix~\ref{sec:regularity}.

\end{remark}

\begin{proposition}\label{propo:err_remainder}
Assume that $K$ satisfies Assumption~\ref{assumptionK} and for $\mu >2$, $f^{(\mu-\ell)}(0)=0$, for $\ell=1,\dots, \rho$, with $\rho-1\le \mu <\rho$, and $f^{(\mu)}$ is bounded on $[0,T]$. Then, for an arbitrary sequence of time points $0<t_1<\dots<t_n$ with step sizes $\tau_n=t_n-t_{n-1}$, $n\ge1$,   there exist  $\xi_j \in (t_{j-1},t_j)$,   for $j\ge 2$, such that the following bound holds
\begin{align*}
& \Vert c(t_n)-c_n\Vert_D   \le
C \left(
\tau_1^{\alpha+\mu}\max_{0\le t \le t_1} \Vert f^{(\mu)}(t) \Vert_B  + \sum\limits_{j=2}^n  \tau_j^2  \xi_j^{\alpha+\mu-2}\max_{0\le t \le t_j} \Vert f^{(\mu)}(t) \Vert_B \right),
\end{align*}
with $C$ depending on $\alpha$, $\mu$ and the constant in \eqref{boundG}.
\end{proposition}
\begin{proof}
By Lemma \ref{lem:fracderi}, $f$ is equal to
\begin{equation}\label{Taylorfrac}
f(t)= \mathcal{I}^{\mu}[f^{(\mu)}](t) = \frac{1}{\Gamma(\mu)} \left( t^{\mu-1} \ast f^{(\mu)}\right)(t).
\end{equation}
By using the operational notation \eqref{operationalconvo}, we can write 
\begin{align*}
c(t_n) = \left[K(\partial_t)f \right](t_n) &= \int_0^{t_n} k(t_n-s) \int_0^s \frac{(s-\sigma)^{\mu-1}}{\Gamma(\mu)} f^{(\mu)}(\sigma) \,d\sigma \,ds \\[.5em]
&= \int_0^{t_n} k(t_n-s) \int_0^{t_n} \frac{(s-\sigma)_+^{\mu-1}}{\Gamma(\mu)} f^{(\mu)}(\sigma) \,d\sigma \,ds \\[.5em]
&= \int_0^{t_n}\left( \int_0^{t_n} k(t_n-s) (s-\sigma)_+^{\mu-1} \,ds\right)\frac{f^{(\mu)}(\sigma)}{\Gamma(\mu)}  \,d\sigma \\[.5em]
&=\int_0^{t_n}  \left[K(\partial_t)(\cdot-\sigma)^{\mu-1}_+ \right](t_n)\frac{f^{(\mu)}(\sigma)}{\Gamma(\mu)} \,d\sigma
\end{align*}
and
\begin{align*}
c_n = \left[K(\partial^{\Delta}_t)f \right]_n &= \sum_{j=1}^{n} \omega_{n,j} \int_{0}^{t_j} \frac{(t_j-\sigma)^{\mu-1}}{\Gamma(\mu)} f^{(\mu)}(\sigma) \,d\sigma\\&
 = \sum_{j=1}^{n} \omega_{n,j} \int_{0}^{t_n} \frac{(t_j-\sigma)_+^{\mu-1}}{\Gamma(\mu)} f^{(\mu)}(\sigma) \,d\sigma
\\&= \int_{0}^{t_n} \left(\sum_{j=1}^{n} \omega_{n,j} (t_j-\sigma)_+^{\mu-1}\right)  \frac{ f^{(\mu)}(\sigma)}{\Gamma(\mu)}\,d\sigma \\
&= \int_{0}^{t_n}  \left[K(\partial^{\Delta}_t)(\cdot-\sigma)^{\mu-1}_+ \right]_n \frac{ f^{(\mu)}(\sigma)}{\Gamma(\mu)}\,d\sigma.
\end{align*}
Then the error $ \left[K(\partial_t)f \right](t_n)- \left[K(\partial^{\Delta}_t)f \right]_n$ can be written in terms of the gCQ error for the Peano kernel, like
\begin{align*}
c(t_n) - c_n &= \int_{0}^{t_n} \left( K(\partial_t)(\cdot-\sigma)^{\mu-1}_ {+} (t_n) - \left[K(\partial^{\Delta}_t)(\cdot-\sigma)^{\mu-1}_+ \right]_n \right) \frac{ f^{(\mu)}(\sigma)}{\Gamma(\mu)}\,d\sigma.
\end{align*}
Since $\mu-1 > 1$, we can apply Corollary~\ref{coro:peano} and bound
\begin{align*}
 \Vert c(t_n)-c_n\Vert_D  &\le C  \int_{0}^{t_n}\frac{
 \Vert f^{(\mu)}(\sigma) \Vert_B  }{\Gamma(\mu)} \left(\tau_1(\tau_1-\sigma)_+^{\alpha+\mu-2}	+\sum\limits_{j=2}^n\tau_j^2 (\xi_j-\sigma)_+^{\alpha+\mu-3}\right)\,d\sigma \\[.5em]
&= C  \frac{
\max_{0\le t \le \tau_1} \Vert f^{(\mu)}(t) \Vert_B  }{\Gamma(\mu)}\tau_1 \int_{0}^{\tau_1}(\tau_1-\sigma)^{\alpha+\mu-2} \,d\sigma
\\[.5em]
&+ C\sum\limits_{j=2}^n \frac{
\max_{0\le t \le \xi_j} \Vert f^{(\mu)}(t) \Vert_B  }{\Gamma(\mu)} \tau_j^2 \int_{0}^{\xi_j} (\xi_j-\sigma)^{\alpha+\mu-3}\,d\sigma\\[.5em]
&= C  \frac{
\max_{0\le t \le t_1} \Vert f^{(\mu)}(t) \Vert_B  }{\Gamma(\mu)(\alpha+\mu-1)}\tau_1^{\alpha+\mu} + C\sum\limits_{j=2}^n \frac{
\max_{0\le t \le t_j}
f^{(\mu)}(t) \Vert_B  }{\Gamma(\mu)(\alpha+\mu-2)} \tau_j^2  \xi_j^{\alpha+\mu-2}.
\end{align*}
And the result follows.
\end{proof}

We state now the main result of the paper for $f$ satisfying \eqref{dataform}, generalizing the results in \cite{Lu86} to variable steps.
\begin{theorem}\label{thm:errgCQ}
Let $\beta>-1$. Assume that $K$ satisfies Assumption~\ref{assumptionK} and $f(t)$ admits an expansion in fractional powers of the form
\begin{equation}\label{assum_f}
f(t)= \sum_{\ell=0}^{p} t^{\ell+\beta} \frac{f^{(\ell+\beta)}(0)}{\Gamma(\ell+\beta+1)}  + \frac{1}{\Gamma(p+\beta+1)} \left( t^{p+\beta} \ast f^{(p+\beta+1)}\right) (t),
\end{equation}
with $p > 1-\beta$, $p\in \bN$. Assume that $ f^{(p+\beta+1)}$ is   bounded on $[0,T]$, 
Then
\begin{equation}\label{err_genf}
\begin{split}
&\Vert  c(t_n)-c_n\Vert_D  \le  C \sum_{\ell=0}^{p} \left(\tau_1^{\ell+\alpha+\beta}+\sum_{j=2}^n\tau_j^2\xi_j^{\ell+\alpha+\beta-2}\right)\left.\Vert  f^{(\ell+\beta)}(0)  \Vert_B  \right.
\\[.5em]
&+
C \left(
\tau_1^{p+\alpha+\beta+1}\max_{0\le t \le t_1}\left.\Vert f^{(p+\beta+1)}(t) \Vert_B  \right. + \sum\limits_{j=2}^n  \tau_j^2  \xi_j^{p+\alpha+\beta-1}\max_{0\le t \le t_j}\left.\Vert f^{(p+\beta+1)}(t) \Vert_B  \right.\right),
\end{split}
\end{equation}
for  a general time mesh like in \eqref{gentimes},   $\xi_j \in (t_{j-1},t_j)$ with $j\ge 2$ and a constant $C$ depending on $\alpha$, $\beta$ and $M$ in \eqref{boundG}.
\end{theorem}
Notice that the above result improves significantly the regularity requirements of the result in  \cite{LoSau16}. Notice also that for a graded mesh \eqref{gmesh},  Proposition \ref{prop:errgCQ_powt} shows how to choose the grading parameter $\gamma$  in order to achieve convergence of order one, according to the regularity and the number of vanishing fractional moments at 0  of $f$. Detailed interpretations  are presented in Corollary~\ref{Cor:fullorder_gen} and Corollary~\ref{Cor:fullorder_grad} below.
\begin{corollary}\label{Cor:fullorder_gen}
	Let \(\beta > -1\), and assume that \(f(t)\) admits an expansion as in \eqref{assum_f}. The full order convergence
	\[
	\Vert  c(t_n)-c_n\Vert_D  = O(\tau_{\max}),
	\]
	where \(\tau_{\max}\) is the maximum time step in \eqref{gentimes}, can be achieved if \(f(t)\) in \eqref{assum_f} satisfies one of the following conditions:
	\begin{enumerate}[label=(\roman*)]
		\item \(\beta > 1 - \alpha\).  
		\item \(-\alpha \leq \beta \leq 1 - \alpha\) and \(f^{(\beta)}(0) = 0\).
		\item \(-1 < \beta \leq 1 - \alpha\), \(f^{(\beta)}(0) = 0\) and \(f^{(1 + \beta)}(0) = 0\).
	\end{enumerate}
	In particular, convergence of order 1 is proven on an arbitrary time mesh provided that $f\in C^{3}([0,T],B)$, with $f(0)=0$, which corresponds to $\beta=1$, cf. Theorem~\ref{thm:errgCQ}.
\end{corollary}
\begin{corollary}\label{Cor:fullorder_grad}
	Let $\beta>-1$, and assume that $f(t)$ admits an expansion as in \eqref{assum_f}, with   $f^{(\beta)}(0)\ne 0$ and the times $t_n$ belong to the graded mesh \eqref{gmesh}, $1\le n \le N$. Then the error estimate
	\[
	\Vert  c(t_n)-c_n\Vert_D  \leq C \left\{
	\begin{array}{ll}
		T^{\alpha + \beta} N^{-\gamma(\alpha + \beta)}, & \quad \gamma(\alpha + \beta) < 1, \\
		T^{\alpha + \beta} N^{-1}(1 + \log(n)), & \quad \gamma(\alpha + \beta) = 1, \\
		T^{\alpha + \beta} N^{-1} t_n^{\alpha + \beta - 1/\gamma}, & \quad \gamma(\alpha + \beta) > 1,
	\end{array}
	\right.
	\]
	holds, where \(\gamma\) is the grading parameter in \eqref{gmesh}.
	\end{corollary}

\section{Application to linear subdiffusion equations}\label{sec:fracdiffeq}
We consider the application of the gCQ method to fractional diffusion equations of the form
\begin{equation}\label{fracdiffusion}
\begin{array}{rcl}
	D_t^\alpha u(t) +Au(t) & = & f(t),	\quad t\in (0,T],\\
	u(0) &=& u_0,
\end{array}
\end{equation}
where $A$ is a symmetric, positive definite, elliptic operator,  $0<\alpha<1$, and $D_t^\alpha$ denotes the Caputo fractional derivative of order $\alpha$
\begin{equation}
	\nonumber
	D_t^\alpha u(t) := \mathcal{I}^{1-\alpha}\left[\frac{d}{dt}u \right](t) = \frac{1}{\Gamma(1-\alpha)}\int_{0}^{t}(t-s)^{-\alpha}\frac{d}{ds}u(s)ds.
\end{equation}
Problem \eqref{fracdiffusion} is considered as an abstract initial value problem on a Banach space $B$, for instance $B=L^p(\Omega)$, with $\Omega \subset \bR^d$ a domain, together with Dirichlet or Neumann boundary conditions. This  is the analytical framework in \cite{CuLuPa} and corresponds to take $D=B$ in Assumption~\ref{assumptionK}. We notice that in \cite{CuLuPa} an expansion in fractional powers of the solution $u$ is derived, of the same type as the one we consider in \eqref{assum_f} for the data $f$, under suitable space-regularity conditions on $u_0$ and $f$. 
The ambient space $B$ can also be a finite dimensional vector space if the method of lines is a applied to \eqref{fracdiffusion} and a Galerkin method is used for the spatial discretization. In this case the resolvent estimates below follow in essentially the same way due to \cite{BaThoWah2003}, with several applications so far in the literature, such as \cite{McSloTho2006}.

Notice that when $u(0)=0$, the Caputo fractional derivative $D_t^\alpha u(t)$, with $0<\alpha<1$, is equivalent to Lubich's operational notation $\partial_t^\alpha u(t)$, and
\begin{equation*}
D_t^\alpha u(t)=\partial_t^\alpha u(t)=\partial_t^{\alpha-1} \partial_t u(t)=\partial_t^{1} \partial^{\alpha-1}_t u(t).
\end{equation*}
The Laplace transform of the solution to \eqref{fracdiffusion} is given by
\begin{align*}
	U(z)&=z^{\alpha-1}(z^\alpha I+A)^{-1}u_0 + (z^\alpha I+A)^{-1}F(z),
\end{align*}
with $F(z)$ the Laplace transform of the source $f(t)$. By using the operational notation for convolutions in \cite{Lu88I}, see Remark~\ref{remark:equivalence_gCQ}, we can write the solution as
\begin{equation}\label{fracdiff_exasol}
u(t)= E(t)  u_0+ (K(\partial_t)f)(t),	
\end{equation}
with
\begin{equation}\label{LTpart}
E(t)= \mathcal{L}^{-1}[z^{\alpha-1}(z^\alpha I+A)^{-1}](t)
\end{equation}
and
\begin{equation}\label{fracresolvent}
K(z) = (z^\alpha I+A)^{-1}.
\end{equation}
We denote further
\begin{equation}\label{k_complex}
k(t)=\mathcal{L}^{-1}[K](t),	
\end{equation}
which satisfies $(k\ast f)(t)=(K(\partial_t)f)(t)$.

Let us assume first that $u_0\in D(A)$. Then, we can use the identity
\begin{equation*}
	(z^\alpha I+A)^{-1}z^{\alpha}=I-(z^\alpha I+A)^{-1}A
\end{equation*}
to write
\begin{equation}\label{Eu_0}
	E(t) u_0=u_0-K(\partial_t)Au_0.
\end{equation}
Combining \eqref{fracdiff_exasol} with \eqref{Eu_0} yields
\begin{equation}\label{solutionv}
	v(t):=u(t)-u_0=\left(K(\partial_t) (f-Au_0)\right)(t),
\end{equation}
which is the solution of
\begin{equation}\label{homo_fracdiffusion}
\begin{array}{rcl}
	\partial_t^\alpha v(t) +Av(t) & = & f(t)-Au_0,	\quad t\in (0,T].\\
	v(0)&=&0.
\end{array}
\end{equation}

We use now the notation
\[
H(\partial_t^{\Delta})f
\]
for the gCQ discretization of $H(\partial_t)f$, with $H$ the Laplace transform of a convolution kernel $h$, on an arbitrary time grid $\Delta := 0 < t_1< \dots < t_N=T$. Then we apply the gCQ to discretize the fractional derivative, which corresponds to $H(z) = z^{\alpha}$. Notice that this approximation is defined in \cite{LoSau13}, where more general convolutions are considered, not only those satisfying \eqref{realiLT}-\eqref{boundG}. In this way, we obtain for the smooth case $u_0 \in D(A)$, the scheme
\begin{equation}\label{homo_gCQ}
\left[ \left(\partial^{\Delta}_{t}\right)^{\alpha} v \right]_{n}  + A v_n  =  f(t_n)-Au_0, \qquad n\ge 1,
\end{equation}
with $v_n$ denoting the approximation to $v(t_n)$.

Since the gCQ preserves the composition rule of the continuous convolution, the approximations $v_n $ defined in \eqref{homo_gCQ} are the same as the ones obtained from applying the gCQ in \eqref{solutionv}, this is, it holds
\begin{equation}\label{solv_gcq}
v_n= \left[ K(\partial^{\Delta}_{t})(f-Au_0) \right]_n.
\end{equation}
It then follows that the approximation properties of the time discretization in \eqref{homo_gCQ} depend on the behavior of $K(z)$ and the regularity of $f-Au_0$. In the next Theorem we show that if $-A$ is a sectorial operator and the sectorial angle $\delta$ can be taken arbitrarily close to zero, then $k$ in \eqref{k_complex} satisfies \eqref{realiLT}-\eqref{boundG}. This is for instance the case if $-A=\Delta$ is the Laplacian operator with Dirichlet or Neumann boundary conditions,  with $-A$ acting on a domain $\Omega \subset \bR^d$, and in the $L^p(\Omega)\to L^p(\Omega)$ norm, for any $1\le p \le \infty$, see for instance \cite{Henry,Pazy}. This property is often preserved by standard spatial discretizations, such as finite differences \cite{Ashyralyev} or the finite element method \cite{BaThoWah2003}.
\begin{theorem}\label{thm:k(t)}
Let us assume that for any $\delta>0$, there exists $M(\delta)>0$ such that
\begin{equation}\label{Asectorial}
\|(zI +A)^{-1}\| \le M(\delta)  |z|^{-1}, \quad \mbox{ for }\  |\arg z| < \pi - \delta.
\end{equation}
Then, for $0<\alpha<1$, $k(t)$ in \eqref{k_complex} can be represented by
\begin{equation}\label{realintk}
k(t) = \int_{0}^{\infty} G(x)  \e^{-xt}\,dx,
\end{equation}
with
\begin{equation}\label{Gsubd}
G(x)=\frac{\sin(\pi\alpha)}{\pi}x^\alpha\left({x^\alpha\e^{-i \alpha\pi} I+A}\right)^{-1} \left({x^\alpha\e^{i \alpha\pi} I+A}\right)^{-1},
\end{equation}
which satisfies the bound
\begin{equation}\label{boundGsubd}
\| G(x) \| \le \frac{\sin(\pi\alpha)}{\pi} \left(M((1-\alpha)\pi)\right)^2x^{-\alpha}.
\end{equation}
\end{theorem}
\begin{proof}
By \eqref{Asectorial}, for any $\delta > 0$ and $0<\alpha<1$, it holds
\[
\|(z^{\alpha}I +A)^{-1}\| \le M((1-\alpha)\pi + \alpha \delta ) |z|^{-\alpha}, \quad \mbox{ for }\quad  |\arg z| < \pi-\delta.
\]
The estimate above allows to pass to the limit as $\delta \to 0$, since the constant remains bounded. This yields the bound
\[
\|(z^{\alpha}I +A)^{-1}\| \le M((1-\alpha)\pi )  |z|^{-\alpha}, \quad \mbox{ for }\quad  |\arg z| < \pi.
\]
Then
\[
\lim_{z\to 0,\ |\arg z| < \pi} z(z^{\alpha}I +A)^{-1} = 0
\]
and
\[
\lim_{z\to \infty,\ |\arg z| < \pi} (z^{\alpha}I +A)^{-1} = 0,
\]
with $(z^{\alpha}I +A)^{-1}$ being defined and continuous for $\Im z \ge 0$, and also for $\Im z \le 0$. By \cite[Theorem 10.7d]{Henrici_II} the real inversion formula for the Laplace transform holds and yields
\begin{equation*}
k(t)=\frac{1}{2\pi i}\int_{0}^{\infty} \e^{-xt} \left(\left({(x\e^{-i \pi})^\alpha I+A}\right)^{-1}-\left({(x\e^{i \pi})^\alpha I+A}\right)^{-1} \right)\,dx.	
\end{equation*}
Formula \eqref{Gsubd} now follows from the identity
\begin{align*}
&\left({(x\e^{-i \pi})^\alpha I+A}\right)^{-1}-\left({(x\e^{i \pi})^\alpha I+A}\right)^{-1} \\
& \qquad =((x\e^{i \pi})^\alpha -(x\e^{-i \pi})^\alpha )\left({(x\e^{-i \pi})^\alpha I+A}\right)^{-1}\left({(x\e^{i \pi})^\alpha I+A}\right)^{-1}
\end{align*}
and the bound \eqref{boundGsubd} follows in a straightforward way from \eqref{Asectorial}.
\end{proof}

In this way, the results in Section~\ref{sec:analysis} can be applied in order to choose an optimally graded time mesh for the approximation of $v$, depending on the regularity of the source term $f$.

In the non smooth case $u_0 \notin D(A)$, we can approximate the solution \eqref{fracdiff_exasol} to \eqref{fracdiffusion} by applying the numerical inversion of sectorial Laplace transforms to approximate the term $E(t)u_0$ and the gCQ method to deal with the convolution $K(\partial_t) f$, for which Theorem~\ref{thm:k(t)} remains useful.

\section{Fast and oblivious gCQ for the fractional integral}\label{sec:algfi}
We follow the approach in \cite{BanLo19} and start by addressing the implementation of the gCQ for the fractional integral $\partial_t^{-\alpha}$. Notice that the associated convolution kernel in this case is
\[
k(t)= \frac{t^{\alpha-1}}{\Gamma(\alpha)},
\]
which satisfies assumptions \eqref{realiLT}-\eqref{boundG} with
\begin{equation}\label{Gfi}
G(x) = \frac{\sin(\pi \alpha)}{\pi} x^{-\alpha}.
\end{equation}
This property of $\partial_t^{-\alpha}$ has been used in \cite{BanLo19} in order to develop a special quadrature for the original CQ weights, those corresponding to an approximation with uniform steps. We generalize here the ideas in \cite{BanLo19} in order to obtain a fast and oblivious gCQ algorithm for the fractional integral.

First of all we split the gCQ approximation into a local and a history term like
\begin{equation}\label{gCQshort_split}
\left[\left( \partial^{\Delta}_t \right) ^{-\alpha}f \right]_n  = \sum\limits_{j=1}^{n-n_0}\omega_{n,j} f_j+\sum\limits_{j=\max\left(1, n-n_0+1\right)}^{n}\omega_{n,j} f_j = I_n^{his}+I_n^{loc},
\end{equation}
for a fixed moderate value of $n_0$, which can be also equal to 1, in principle. We now apply different methods to compute $I_n^{his}$ and $ I_n^{loc}$.

For the local term $I_n^{loc}$ we use the original derivation of the gCQ method in \cite{LoSau13}, and the expression \eqref{originalgCQ} for the gCQ weights. This leads to the following complex-contour integral representation
\begin{equation}\label{In_loc}
I_n^{loc}=\frac{1}{2\pi i}\int_{\mathcal{C}} z^{-                                                                                                                                                                                                                                                                                                                                                                                                                                                                                                                                                                                                                                                                                                                                                                                                                                                                                                                                                                                                                                                                                                                                                                                                                        \alpha}\sum\limits_{j=1+(n-n_0)_+}^{n}\tau_j\left(\prod\limits_{\ell=j}^n\frac{1}{1-\tau_\ell z} \right)f_j\,dz,\quad  n\ge 1,
\end{equation}
where $\mathcal{C}$ is a clockwise oriented, closed contour contained in $\Re(z)>0$, which encloses  all poles of the integrand, namely $\tau^{-1}_j$, with $j=1+(n-n_0)_+,\dots,n$. We denote
\[
u^{loc}_{n}(z):=\sum\limits_{j=1+(n-n_0)_+}^{n}\tau_j\left(\prod\limits_{\ell=j}^n\frac{1}{1-\tau_\ell z}\right)  f_j, \quad n \ge 1,
\]
which satisfies the recursion
\begin{equation}\label{recursion_loc}
u^{loc}_{n}(z)=\frac{1}{1-\tau_{n} z}u^{loc}_{n-1}(z)+\frac{\tau_{n}}{1-\tau_{n} z} f_n, \qquad u^{loc}_{(n-n_0)_+}(x)=0
\end{equation}
and write
\begin{equation}\label{I_elli}
I_n^{loc}=\frac{1}{2\pi i}\int_{\mathcal{C}} z^{-\alpha} u^{loc}_n(z)\,dz.
\end{equation}
For the approximation of \eqref{I_elli} we set $m$ the minimum pole of $u^{loc}_{n}(z)$, $M$ the maximum one, and
\[
q:=M/m.
\]
Then we choose $\mathcal{C}$ as the circle $\mathcal{C}_M$ of radius $M$ centered at $M+m/10$. The parametrization of $\mathcal{C}_M$ and the  quadrature nodes and weights for approximating \eqref{I_elli} are obtained in the following way:

\begin{enumerate}[label=(\roman*)]
\item  If $q < 1.1$, we use the standard parametrization of $\mathcal{C}_M$ and apply the composite trapezoid rule.  Notice that, in this case, the time grid $\Delta$ is close to uniform and thus the special quadrature based on Jacobi elliptic functions from \cite{LoSau15apnum} is not required. The threshold value $q=1.1$ has been found heuristically. In this case, the parametrization reads
\begin{equation*}
	\widetilde \gamma_M(\theta)= m/10+ M(1+\e^{i\theta}),\quad \theta\in(-\pi,\pi),
\end{equation*}
and the quadrature nodes and weights are given by
\begin{align}\label{circle_quaw}
	z_l=\widetilde \gamma_M(\theta_l),\quad w_l=-\frac{\widetilde\gamma_M'(\theta_l)}{iN_Q^{loc}}, \quad \mbox{with} \quad \theta_l=-\pi+\frac{2\pi l}{N_Q^{loc}},
\end{align}
for $l=1$, $\cdots$, $N^{loc}_Q$.

\item  If $q\ge 1.1$, we apply the quadrature proposed in \cite{LoSau15apnum} which is based on the special parameterization of $\mathcal{C}_M$
\begin{equation}\label{contour_elli}
	\widetilde \gamma_M(\sigma)=m/10 + \frac{M}{q-1}\left(\sqrt{2q-1}\frac{\lambda^{-1/2}+{\rm{sn}}(\sigma,\lambda)}{\lambda^{-1/2}-{\rm{sn}}(\sigma,\lambda)}-1\right),
\end{equation}
where $\rm{sn}$ is the elliptic sine function. The quadrature nodes $z_l$ and weights $w_l$ in this case are given by
 \begin{equation}\label{elli_quaw}
 	z_l=\widetilde \gamma_M(\sigma_l),\quad w_l=\frac{4K(\lambda)}{2\pi i N_Q^{loc}}\widetilde \gamma_M'(\sigma_l),
 \end{equation}
for $l=1$, $\cdots$, $N^{loc}_Q$.
\end{enumerate}
In any case, we approximate
\begin{equation}\label{Iloc_quad}
I^{loc}_n \approx \sum_{l=1}^{N_{Q}^{loc}} w_l z_l^{-\alpha} u_n^{loc}(z_l).
\end{equation}

For the history term  we use the real integral representation
\begin{equation}\label{I_his}
I_n^{his} =\sum\limits_{j=1}^{n-n_0}\omega_{n,j} f_j
=\frac{\sin(\pi\alpha)}{\pi}\int_0^\infty x^{-\alpha}\left(\prod\limits_{\ell=n-n_0+1}^{n}r(\tau_\ell x)\right)y_{n-n_0}(x)\,dx,
\end{equation}
with $r(x)$ defined in \eqref{gcqw} and $y_n$ defined in \eqref{eulersol}. To approximate the integral, we generalize the quadrature method from \cite{BanLo19} to accommodate variable time grids, yielding corresponding nodes \(x_l\) and weights \(\varpi_l\). We denote the error tolerance of this new quadrature as \(\tol\)  and denote \(N^{his}_Q\) the total number of quadrature nodes, which depends on both the time grid and \(\tol\). Then \eqref{I_his} is approximated by
\begin{equation}\label{Ihis_quad}
I_n^{his}\approx \sum\limits_{l=1}^{N^{his}_Q} \varpi_l\left(\prod\limits_{j=n-n_0+1}^{n}r(\tau_j x_l)\right) y_{n-n_0}(x_l)=: \widetilde{I}_n^{his}.
\end{equation}
with each $y_{n-n_0}(x_l)$ satisfying the recursion
\begin{equation}\label{recursion_his}
	y_{n}(x_l)=\frac{1}{1+\tau_{n} x_l}y_{n-1}(x_l)+\frac{\tau_{n}}{1+\tau_{n} x_l}f_{n} \ ; \qquad y_0(x_l)=0, \quad 1\le l\le N_Q^{his}.
\end{equation}
In this way the evaluation of $I_n^{his}$ for $1\le n\le N$ requires a number of operations proportional to $N$ and storage proportional to $N_Q^{his}$.
We finally approximate
\begin{equation*}
\left[\partial_t^{-\alpha}f \right]_n \approx \widetilde{I}_n^{his}+\widetilde{I}_n^{loc}, \qquad   n \ge 1.
\end{equation*}

In our numerical experiments, we set
\begin{equation*}
 n_0=\min\left(10,N\right),\quad N^{loc}_Q=\max\left(50,n_0^2\right),
\end{equation*}
for the evaluation of $\widetilde{I}_n^{loc}$.  In this way the total complexity of this algorithm is $O(n_0N_Q^{loc}+\left(N-n_0\right)N_Q^{his})$, while the memory requirements grow like $O(N_Q^{loc}+N_Q^{his})$. For  the uniform step approximation of \eqref{fracint}, this is, the original CQ, it is proven in \cite{BanLo19} that $ N_Q^{his} =O(\log(N))$. The results reported in Tables~\ref{tab:Quad_FracInt}-\ref{tab:PDEu01_Quad} for the generalization of this algorithm to variable steps, indicate that also on graded meshes like \eqref{gmesh} it is $N_Q^{his} = O(\log N)$, which certainly is an very important feature of this method. A rigorous theoretical analysis of the required complexity and memory requirements of the generalized fast and oblivious algorithm is beyond the scope of the present paper.

\section{Fast and oblivious gCQ for linear subdiffusion equations}\label{sec:algsubdiff}
Here we show how to adapt the algorithm described in the previous section to the efficient computation of $v_n$ in \eqref{homo_gCQ}. In order to apply the fast gCQ algorithm for the fractional derivative we observe that, by definition,
\[
\partial_t^{\alpha} =  \partial_t^{\alpha-1} \partial_t.
\]
Then, by the discrete composition rule, it also holds
\[
\left(\partial^{\Delta}_t \right)^{\alpha} =  \left(\partial^{\Delta}_t \right)^{\alpha-1} \partial^{\Delta}_t,
\]
where $\partial^{\Delta}_{t}$ denotes the approximation of the first order derivative provided by the gCQ method, this is,
\[
(\partial^{\Delta}_{t} v)_n = \frac{v_n-v_{n-1}}{\tau_n}, \quad 1\le n\le N.
\]
Thus, we can write \eqref{homo_gCQ} like
\[
\sum_{j=1}^n \omega_{n,j}  \left[\partial^{\Delta}_{t} v \right]_{j}  + A v_n  =  f(t_n)-Au_0, \qquad n\ge 1.
\]
with $ \omega_{n,j} $ the gCQ weights \eqref{gcqw} for the fractional integral $ \partial_t^{\alpha-1}$, this is with
	\[G(x)= \frac{\sin(\pi (1-\alpha))}{\pi} x^{\alpha-1}.\]
Subsequently, we have
\[
\omega_{n,n}  \left[\partial^{\Delta}_{t} v \right]_{n}  + A v_n  =  f(t_n)-Au_0-\sum_{j=1}^{n-1} \omega_{n,j}  \left[\partial^{\Delta}_{t} v \right]_{j}, \qquad n\ge 1.
\]
Noticing that from \eqref{originalgCQ} we have $\omega_{n,n} = K[\tau_n^{-1}] = \tau_n^{-\alpha+1}$, the scheme reads
\begin{equation}\label{scheme_subd}
\left( \tau_n^{-\alpha}I + A \right) v_n=  f(t_n)-Au_0 + \tau_n^{-\alpha} v^{n-1}  -\sum_{j=1}^{n-1} \omega_{n,j}  \left[\partial^{\Delta}_{t} v \right]_{j}, \qquad n\ge 1.
\end{equation}
The summation in the right hand side can be efficiently evaluated by applying a variation of the algorithm in Section~\ref{sec:algfi}. As then, we separate this term into a local and a history term, by
\[
\sum_{j=1}^{n-1} \omega_{n,j}  \left[\partial^{\Delta}_{t} v \right]_{j}=\sum_{j=1}^{n-n_0} \omega_{n,j}  \left[\partial^{\Delta}_{t} v \right]_{j} + \sum_{j=\max(1,n-n_0+1)}^{n-1} \omega_{n,j}  \left[\partial^{\Delta}_{t} v \right]_{j}.
\]
After applying the corresponding quadratures, as described in Section~\ref{sec:algfi}, we now  approximate the history term by
\[
\sum_{j=1}^{n-n_0} \omega_{n,j}  \left[\partial^{\Delta}_{t} v \right]_{j} \approx \sum_{l=1}^{N_{Q}^{his}} \varpi_l\left(\prod\limits_{j=n-n_0+1}^{n}r(\tau_j x_l)\right) y_{n-n_0}(x_l),
\]
with the values of $f_j$ in \eqref{recursion_his} replaced by $ \left[\partial^{\Delta}_{t} v \right]_{j}$. The approximation of the local term now reads, after quadrature,
\[
\sum_{j=\max(1,n-n_0+1)}^{n-1} \omega_{n,j}  \left[\partial^{\Delta}_{t} v \right]_{j} \approx \sum_{l=1}^{N_{Q}^{loc}} w_l z_l^{\alpha-1} \frac{u_{n-1}^{loc}(z_l)}{1-\tau_n z_l},
\]
with $u_{n-1}^{loc}$ as in \eqref{recursion_loc}, with $f_{n-1}$ replaced by $ \left[\partial^{\Delta}_{t} v \right]_{n-1}$. The complexity and memory requirements are as explained in Section~\ref{sec:algfi}.

\section{Numerical experiments}\label{sec:experiments}

We test in this section the convergence results proven in Proposition \ref{prop:errgCQ_powt} and Section~\ref{sec:fracdiffeq}, by applying the gCQ method to different examples. Since the main goal of this paper is to derive accurate error estimates for the gCQ, we set a rather high tolerance requirement for the quadratures described in the previous section. More precisely, we set $\tol=10^{-8}$ for the evaluation of $\widetilde{I}_n^{his}$, for all $1\le n\le N$ and all values of $N$.
 \subsection{Computation of the fractional integral}\label{subsec:expfi}
To verify Proposition \ref{prop:errgCQ_powt}, we consider the approximation to the fractional integral \eqref{fracint} with
\begin{equation*}
	f(t)=t^\beta.
\end{equation*}
The exact solution is then given by
\begin{equation}\label{fracint_exact}
	\mathcal{I}^\alpha[f](t)=\frac{\Gamma(\beta+1)}{\Gamma(\alpha+\beta+1)}t^{\alpha+\beta}.
\end{equation}
Set $T=1$.  The maximum absolute error is measured by
\begin{equation*}
	\max_{1\le n\le N}\left|\mathcal{I}^{\alpha}[f](t_n) - \mathcal{I}^{\alpha}_n[f] \right|.
\end{equation*}
For  $\alpha=0.8$, the value of $N_Q^{his}$ in \eqref{Ihis_quad}  is listed in Table~\ref{tab:Quad_FracInt}, which shows that the quadrature number $N_Q^{his}$  increases logarithmically with respect to  $N$.

 In Figure~\ref{fig:fracint_order1}, we present the maximum absolute error for different values of $\alpha$ and $\beta$. The numerical results demonstrate that the gCQ method \eqref{gCQ}  can achieve first order convergence with $\gamma=1/(\alpha+\beta)$, which is consistent with the theoretical results. Further,
 in Figure~\ref{fig:fracint_order_gam}, we test the relation between the convergence order  and the grading parameter. We can observe that the convergence order is $\min \left(1,\gamma(\alpha+\beta)\right)$, in agreement with Proposition \ref{prop:errgCQ_powt}. In addition, the evolution of the error on different graded meshes, as presented in Figure~\ref{fig:fracint_error}, shows that the error near the origin is much smaller for the nonuniform mesh than for the uniform one.
 \begin{table}[H]
	\renewcommand{\captionfont}{\small}
	\centering
	\small
	\begin{spacing}{0.8}
		\caption{ Number of $N_Q^{his}$  used to  approximate \eqref{fracint_exact} on  different graded meshes with  $\alpha=0.8$,   $n_0=10$, $\tol=10^{-8}$.}
		\label{tab:Quad_FracInt}
	\end{spacing}
	\begin{tabular}{p{1cm}<{\centering}p{1cm}<{\centering}p{1cm}<{\centering}p{1cm}<{\centering}p{1cm}<{\centering}p{1cm}<{\centering}p{1cm}<{\centering}p{1cm}<{\centering}<{\centering}p{1cm}<{\centering}p{1cm}<{\centering}}
		\hline \specialrule{0pt}{2pt}{2pt}
		$\gamma$\textbackslash$N$  & $16$ & $32$ & $64$ & $128$ & $256$ & $512$ & $1024$ &$2048$ & $4096$  \\ \specialrule{0pt}{1.5pt}{1.5pt}\hline		\specialrule{0pt}{2pt}{2pt}
		2  & 29 & 46 & 55 & 63 & 72 & 77 & 85 & 89 & 94  \\ \specialrule{0pt}{1.5pt}{1.5pt}
		4  & 52 & 93 & 118 & 139 & 158 & 174 & 191 & 204 & 217\\
		\specialrule{0pt}{1.5pt}{1.5pt}
		6  & 74 & 141 & 185 & 222 & 253 & 283 & 313 & 341 & 365\\
		\specialrule{0pt}{1.5pt}{1.5pt}
		8& 90 & 189 & 255 & 313 & 359 & 405 & 451 & 493 & 535\\
		\specialrule{0pt}{1.5pt}{1.5pt}
		10 & 104 & 237 & 329 & 407 & 474 & 540 & 604 & 667 & 724\\\specialrule{0pt}{2pt}{2pt}
		\hline
	\end{tabular}
\end{table}


\begin{figure}[H]
\begin{center}
 \includegraphics[height=6cm,width=6.8cm]{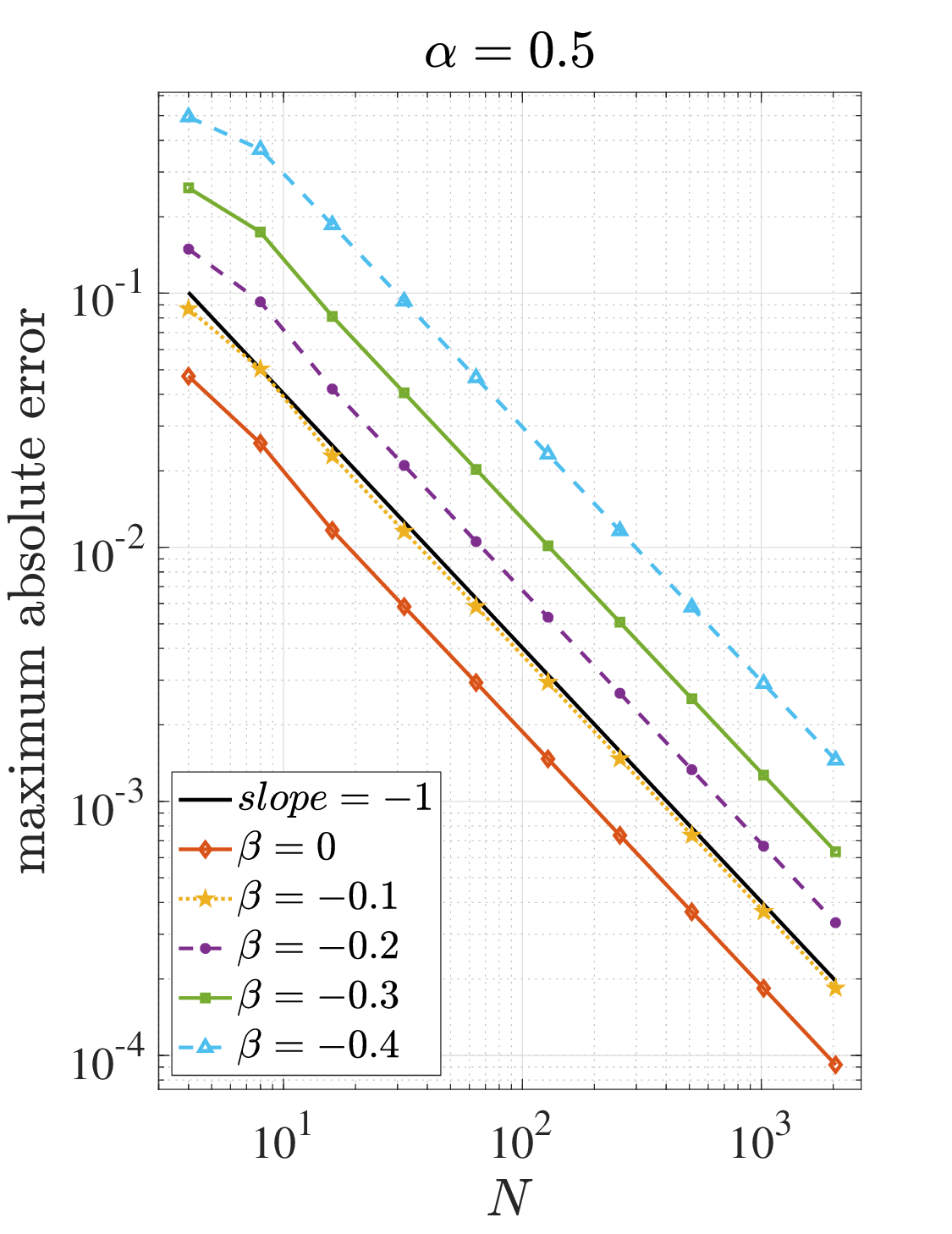}
 \includegraphics[height=6cm,width=6.8cm]{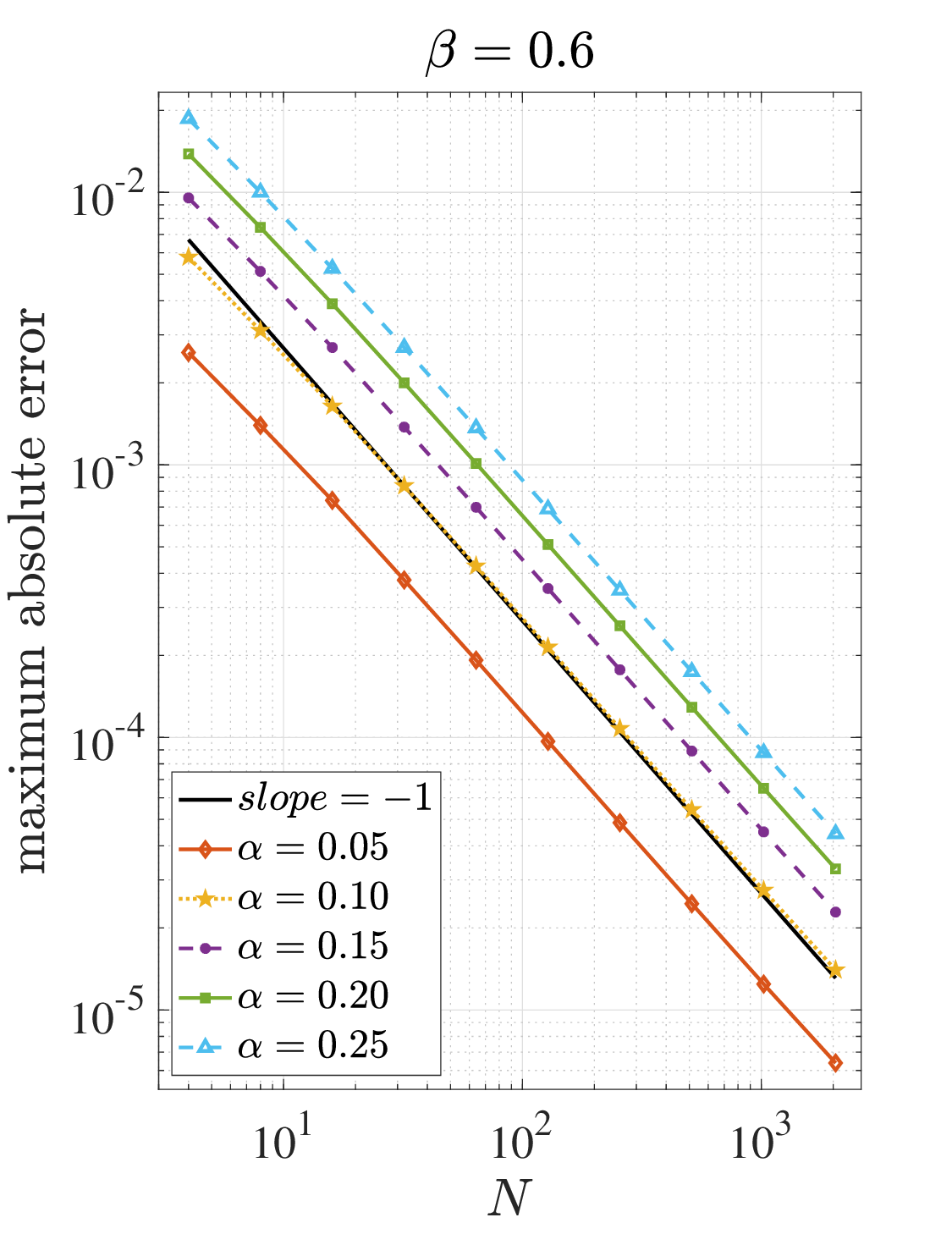}
 \caption{Maximum absolute error in the approximation of \eqref{fracint_exact}  with $\gamma=\frac{1}{\alpha+\beta}$.}
 \label{fig:fracint_order1}
\end{center}
    \end{figure}

\begin{figure}[H]
\centering
 \includegraphics[height=6cm,width=6.8cm]{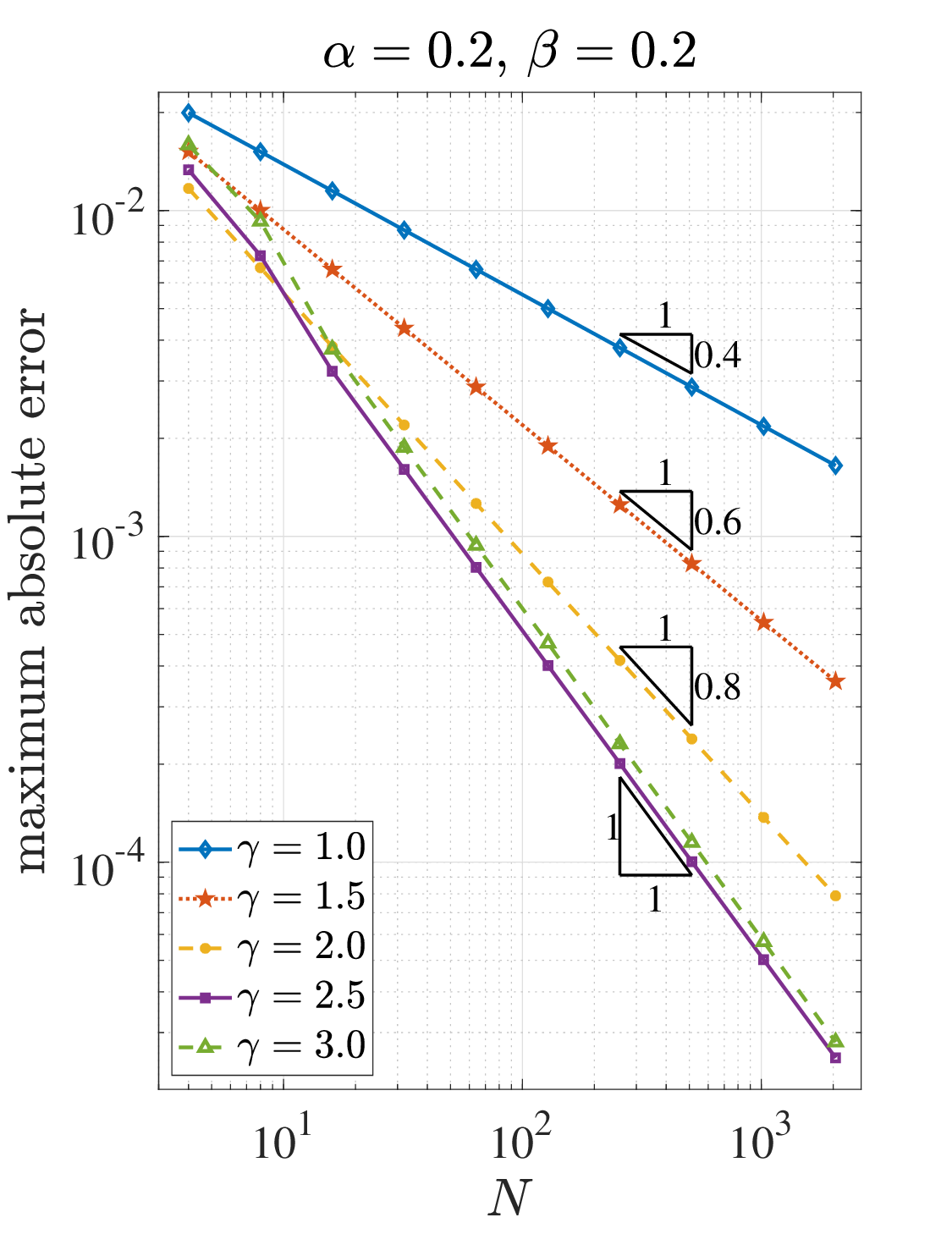}
 \includegraphics[height=6cm,width=6.8cm]{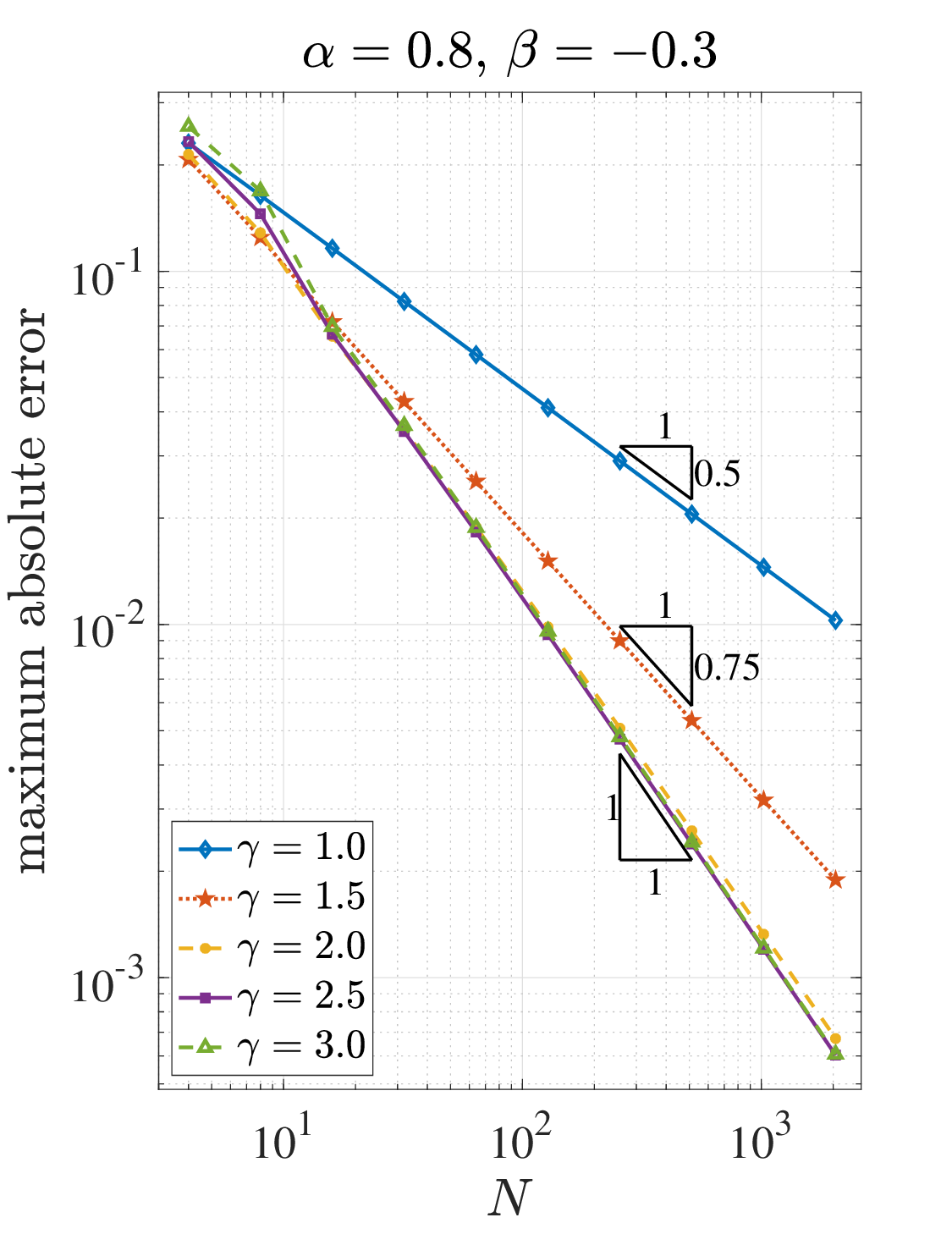}
\caption{Maximum absolute error in the approximation of \eqref{fracint_exact}.}
 \label{fig:fracint_order_gam}
    \end{figure}

\begin{figure}[H]
\centering
\includegraphics[height=6cm,width=6.8cm]{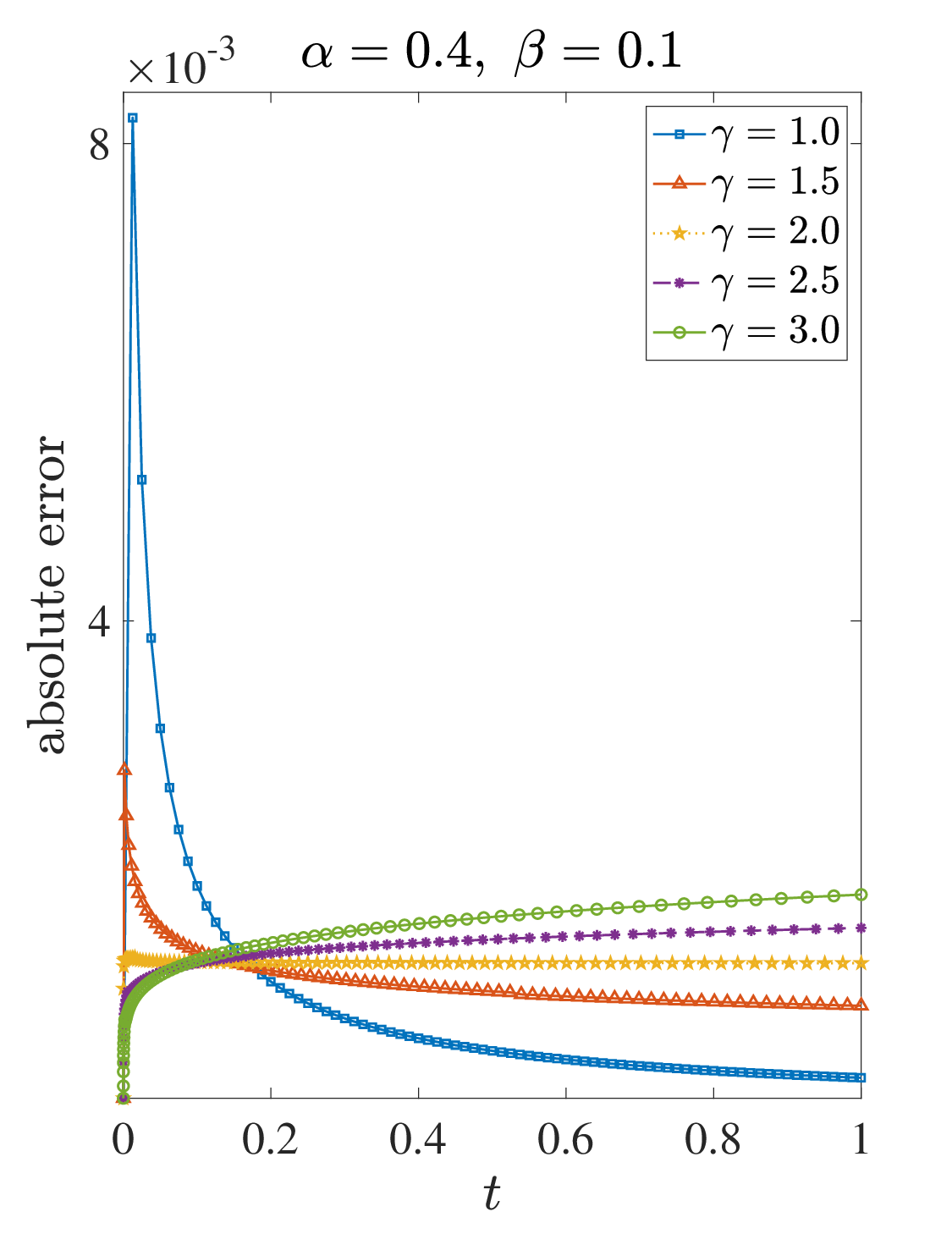}
\includegraphics[height=6cm,width=6.8cm]{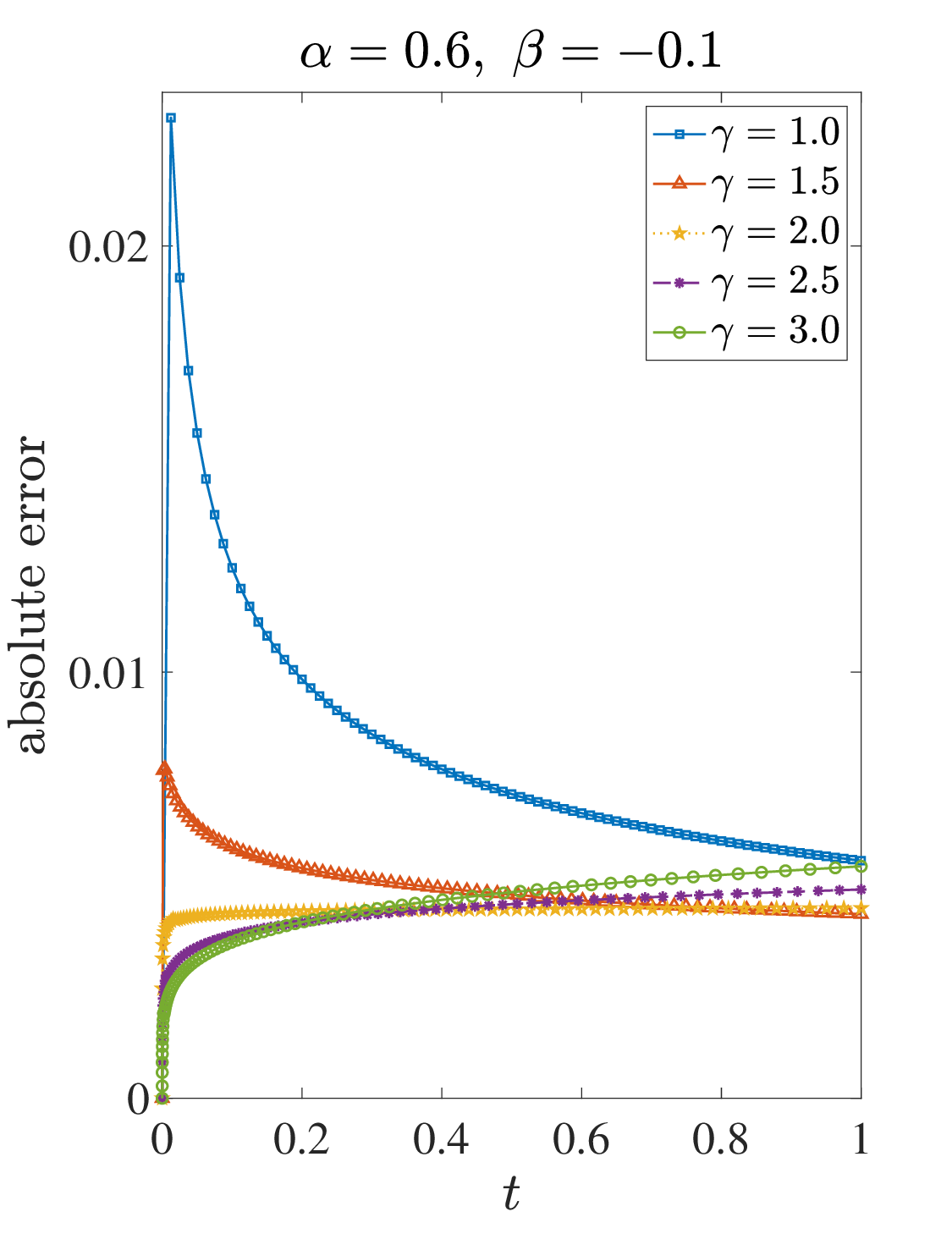}
\caption{Absolute error in the approximation of \eqref{fracint_exact}  with  $N=80$.}
\label{fig:fracint_error}
\end{figure}

\subsection{Fractional diffusion equations with non smooth data}\label{subsec:expFDE}

\begin{example}\label{ex_ODE}
We consider the	fractional ordinary diffusion equation
\begin{equation*}\label{sca_pro}
	D_t^\alpha u(t)+ u(t)=f(t),	\quad u(0)=0, \quad 0<t\leq
	1,
\end{equation*}
  with an exact solution
\begin{equation*}
	u(t)=t^\beta.
\end{equation*}
The corresponding source term is
\begin{equation*}
\label{source}
f(t)=\frac{\Gamma(\beta+1)}{\Gamma(\beta-\alpha+1)}t^{\beta-\alpha}+t^\beta.
\end{equation*}
\end{example}
The results in Table~\ref{tab:ODE_Quad} demonstrate that the number of quadrature points $N_Q^{his}$ increases slowly  with respect to $N$ on the graded mesh \eqref{gmesh}. In Figure~\ref{fig:ODEorder}, we present the maximum errors for different combinations of $\alpha$ and $\beta$, revealing the method's consistent full-order convergence of one on an optimal graded mesh, as discussed in Section \ref{sec:fracdiffeq}. Additionally, Figure~\ref{fig:ODEorder_gam} visually provides the convergence rate of gCQ method on various graded meshes, with the uniform mesh serving as a specific case that exhibits a convergence rate of $O(N^{-\beta})$.

\begin{table}[H]
\renewcommand{\captionfont}{\small}
	\centering
	\small
	\begin{spacing}{0.8}
		\caption{Value of $N_Q^{his}$ for Example~\ref{ex_ODE} with  $\alpha=0.5$,   $\tol=10^{-8}$.}
		\label{tab:ODE_Quad}
	\end{spacing}
	\begin{tabular}{p{1cm}<{\centering}p{1cm}<{\centering}p{1cm}<{\centering}p{1cm}<{\centering}p{1cm}<{\centering}p{1cm}<{\centering}p{1cm}<{\centering}p{1cm}<{\centering}<{\centering}p{1cm}<{\centering}p{1cm}<{\centering}}
	\hline \specialrule{0pt}{2pt}{2pt}
	$\gamma$\textbackslash$N$  & $16$ & $32$ & $64$ & $128$ & $256$ & $512$ & $1024$ &$2048$ & $4096$  \\ \specialrule{0pt}{1.5pt}{1.5pt}\hline		\specialrule{0pt}{2pt}{2pt}
2 & 33 & 50 & 62 & 72 & 83 & 91 & 98 & 104 & 113 \\                \specialrule{0pt}{1.5pt}{1.5pt}
4 & 62 & 107 & 137 & 165 & 191 & 216 & 240 & 264 & 288 \\       \specialrule{0pt}{1.5pt}{1.5pt}
6 & 91 & 168 & 227 & 277 & 325 & 373 & 421 & 468 & 516 \\     \specialrule{0pt}{1.5pt}{1.5pt}
8 & 119 & 239 & 326 & 404 & 480 & 557 & 634 & 715 & 798 \\  	 \specialrule{0pt}{1.5pt}{1.5pt}
10 & 148 & 314 & 436 & 549 & 660 & 773 & 887 & 1007 & 1130 \\ 	\specialrule{0pt}{2pt}{2pt}
	\hline
\end{tabular}
\end{table}
\begin{figure}[H]
\centering
 \includegraphics[height=6cm,width=6.8cm]{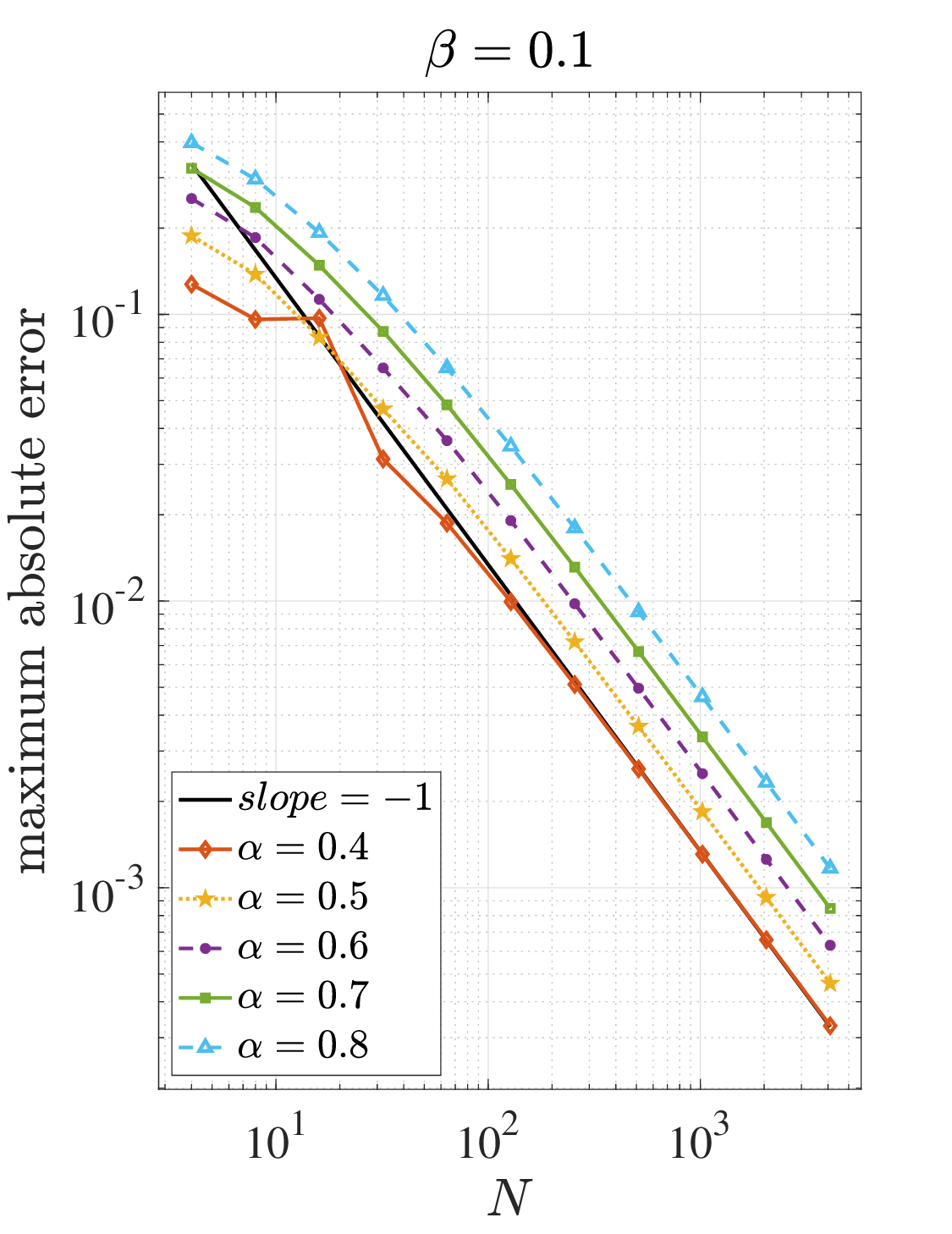}
 \includegraphics[height=6cm,width=6.8cm]{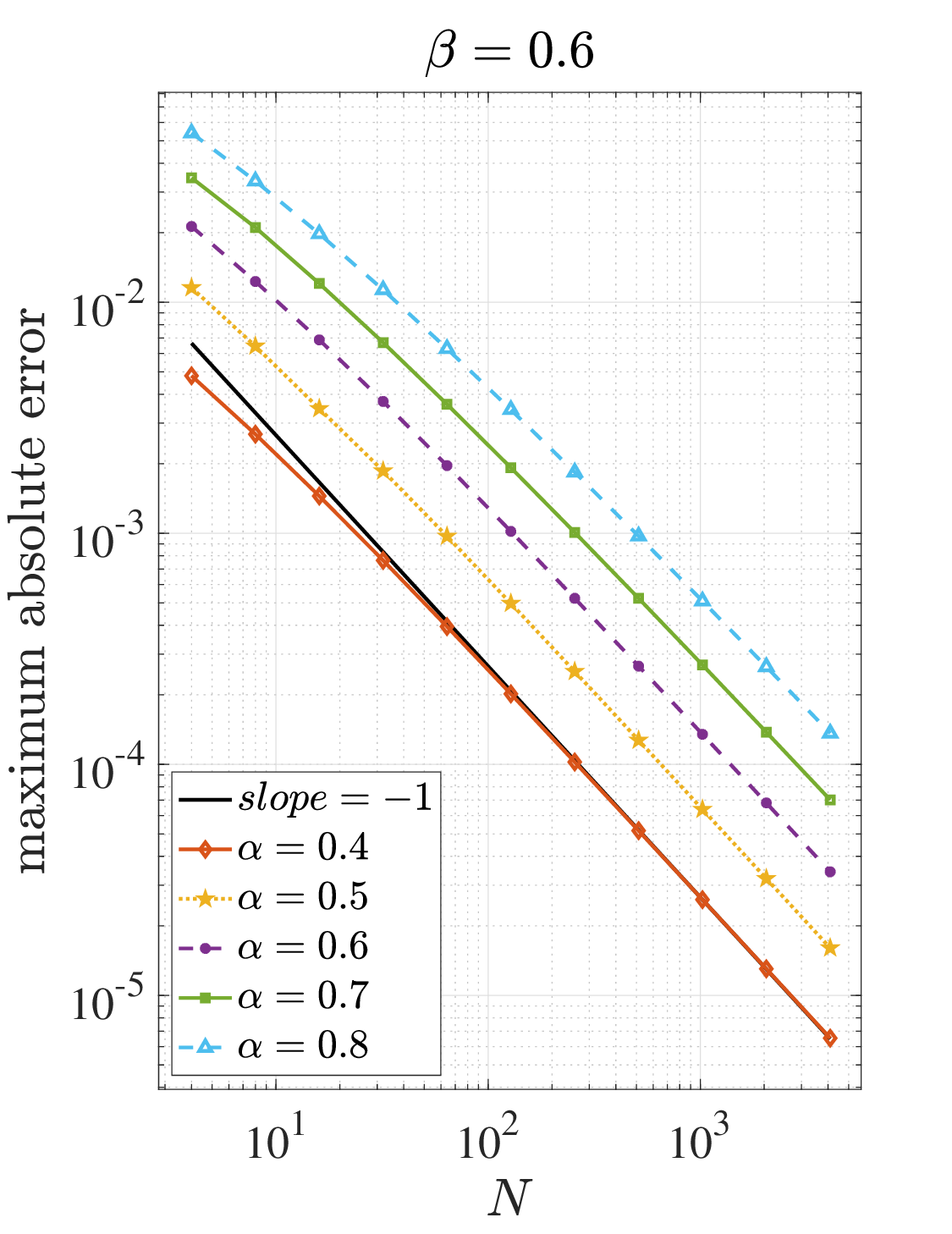}
 \caption{Maximum absolute errors for Example~\ref{ex_ODE} with $\gamma=\frac{1}{\beta}$.}
 \label{fig:ODEorder}
    \end{figure}
\begin{figure}[H]
\centering
 \includegraphics[height=6cm,width=6.8cm]{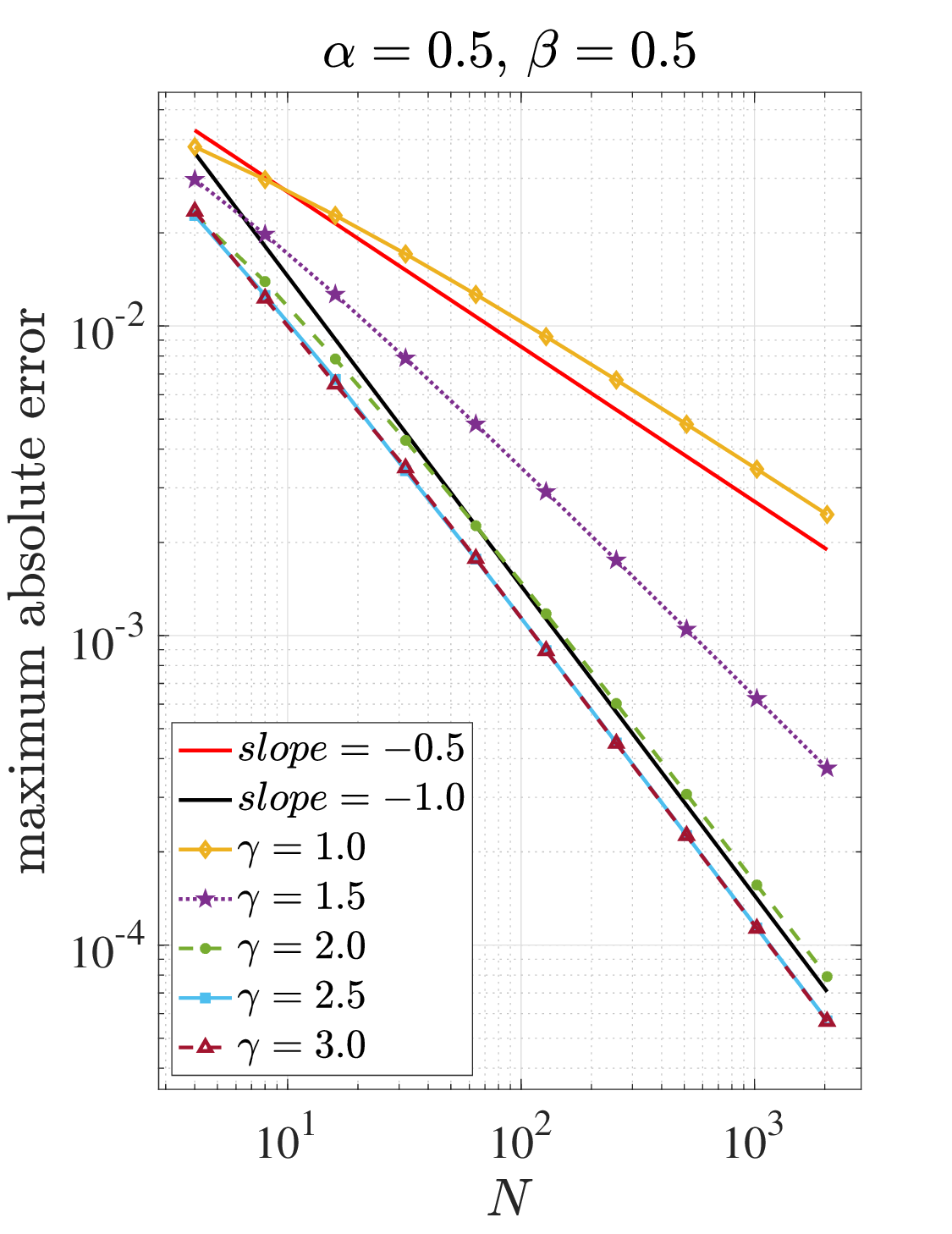}
 \includegraphics[height=6cm,width=6.8cm]{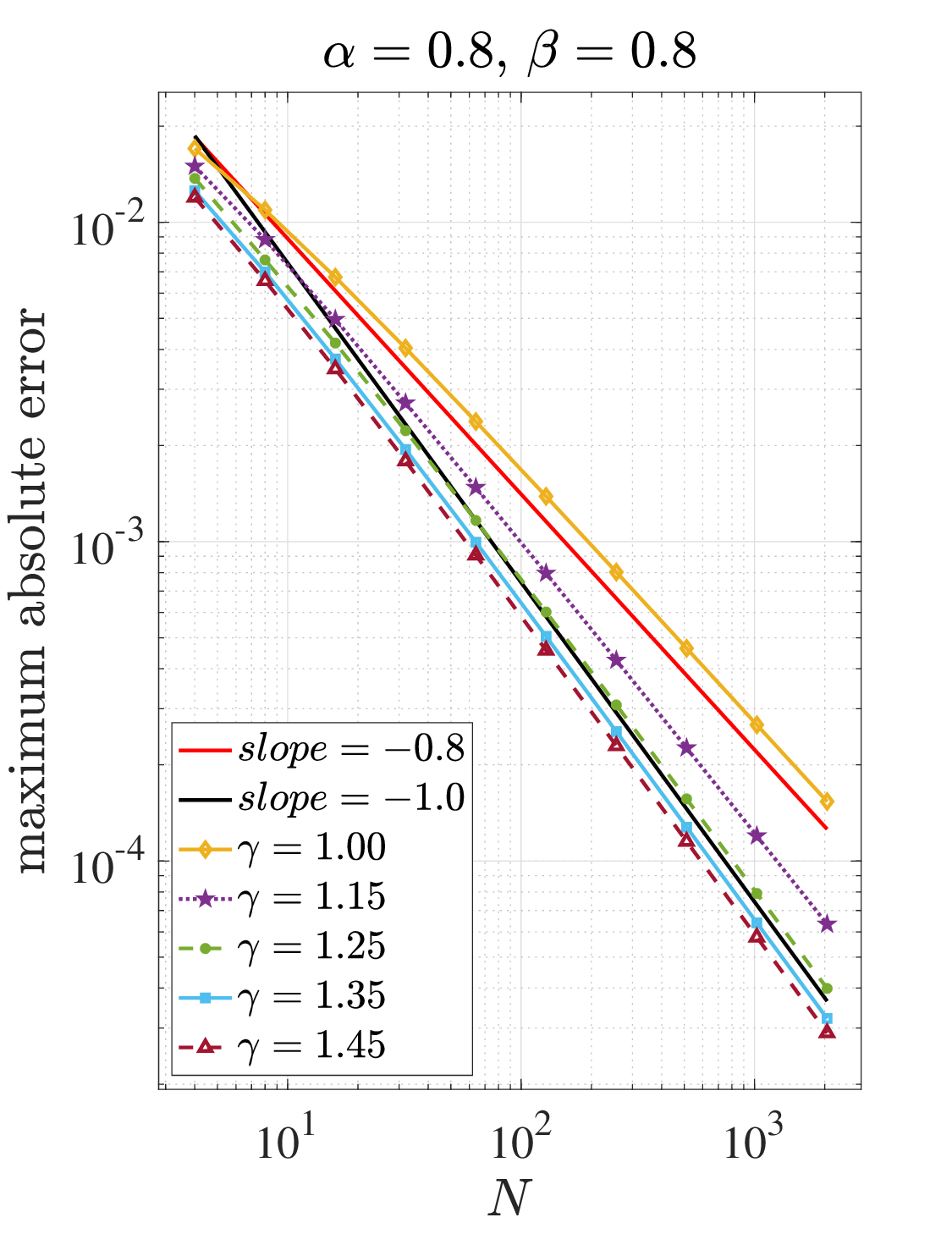}
  \caption{Maximum absolute errors for Example~\ref{ex_ODE} with different values of $\gamma$.}
 \label{fig:ODEorder_gam}
    \end{figure}

\begin{example}\label{ex_PDEsmooth}
	We consider here the following subdiffusion equation on a 1D domain $\Omega=(-1,1)$:
	\begin{equation}
		\left\{\begin{array}{lll}\label{PDE_smooth}
			D_t^\alpha  u-\Delta u=f(x,t),&\quad (x,t)\in \Omega\times(0,1],\\[.5em]
			u(x,0)=\cos(\pi x/2),&\quad x\in\Omega,\\[.5em]
			u(x,t)=0,&\quad (x,t)\in \partial \Omega\times(0,1].
		\end{array}\right.
	\end{equation}
	The exact solution is given by
	\begin{equation}
		\nonumber
		u(x,t)=(1+t^{\beta})\cos(\pi x/2),
	\end{equation}
	and the source term is
	\begin{equation*}
		f(x,t)=\left(\frac{\Gamma(\beta+1)}{\Gamma(\beta-\alpha+1)}t^{\beta-\alpha}+\pi^2(1+t^{\beta})/4\right)\cos(\pi x/2).
	\end{equation*}
\end{example}

Problem \eqref{PDE_smooth} fits the format in \eqref{fracdiffusion}  with $A=-\Delta$ and $u_0(x)=\cos(\pi x/2)$. The domain of $-\Delta$ is given by
\begin{equation}\label{D(A)}
	D(-\Delta)=H_0^1(\Omega)\cap H^2(\Omega)
\end{equation}
and $u_0(x)\in D(-\Delta)$. Then, as discussed in Section~\ref{sec:fracdiffeq}, we can solve the equivalent problem \eqref{homo_fracdiffusion}.

We first discretize \eqref{homo_fracdiffusion} in space by applying the finite element method with piecewise linear basis functions and mesh-width  $\Delta_x$. The resulting discrete Laplacian operator is well known to satisfy the hypotheses of Theorem~\ref{thm:k(t)}. Then, according to the regularity of $f-\Delta u(x,0)$, the convergence order of the gCQ method is one on graded meshes \eqref{gmesh} with $\gamma \ge \frac{1}{\beta}$. In this way we compute the numerical solution $v_{\Delta_x}^n$ of \eqref{homo_fracdiffusion}, and so $u_{\Delta_x}^n=v_{\Delta_x}^n+u(x,0)$. In our experiment, we measure the maximum $L^2$-norm error, which  is defined as
\[\max\limits_{1\leq n\leq N} \Vert u(t_n)-u_{\Delta_x}^n\Vert_{L^2}.\]
In Figure~\ref{fig:PDE_order} we show the expected result for $\gamma=\frac{1}{\beta}$, for different values of $\beta$. We also show in Table~\ref{tab:PDE_Quad} the number of quadrature nodes $N_Q^{his}$ used to approximate  \eqref{homo_gCQ}. In Figure~\ref{fig:PDE_error} we display the error profiles for $\alpha=\beta=0.5$, for both the uniform mesh and the optimal graded one, with $\gamma=2$.  We can see that the  gCQ method on the nonuniform mesh is more accurate near $t=0$.

\begin{table}[H]
	\renewcommand{\captionfont}{\small}
	\centering
	\small
	\begin{spacing}{0.8}
		\caption{Value of $N_Q^{his}$ for Example~\ref{ex_PDEsmooth} with  $\beta=0.2$,  $\gamma=\frac{1}{\beta}$, $\tol=10^{-8}$.}
		\label{tab:PDE_Quad}
	\end{spacing}
	\begin{tabular}{p{1cm}<{\centering}p{1cm}<{\centering}p{1cm}<{\centering}p{1cm}<{\centering}p{1cm}<{\centering}p{1cm}<{\centering}p{1cm}<{\centering}p{1cm}<{\centering}<{\centering}p{1cm}<{\centering}p{1cm}<{\centering}}
		\hline \specialrule{0pt}{2pt}{2pt}
		$\alpha$\textbackslash$N$ & $16$ & $32$ & $64$ & $
		128$ & $256$ & $512$ & $1024$ & $2048$& $4096$   \\ \specialrule{0pt}{1.5pt}{1.5pt}\hline		\specialrule{0pt}{2pt}{2pt}
		0.3 & 90 & 155 & 200 & 238 & 273 & 308 & 340 & 369 & 398 \\
		\specialrule{0pt}{1.5pt}{1.5pt}
		0.4 & 95 & 162 & 212 & 251 & 291 & 328 & 364 & 398 & 436 \\
		\specialrule{0pt}{1.5pt}{1.5pt}
		0.5 & 98 & 167 & 220 & 262 & 306 & 346 & 389 & 429 & 470 \\
	\specialrule{0pt}{1.5pt}{1.5pt}
		0.6 & 101 & 176 & 227 & 276 & 319 & 368 & 411 & 455 & 501 \\
		\specialrule{0pt}{1.5pt}{1.5pt}
		0.7 & 106 & 180 & 234 & 285 & 334 & 382 & 430 & 484 & 534 \\
		\specialrule{0pt}{2pt}{2pt} \hline
	\end{tabular}
\end{table}

\begin{figure}[H]
	\centering
	\includegraphics[height=6cm,width=6.8cm]{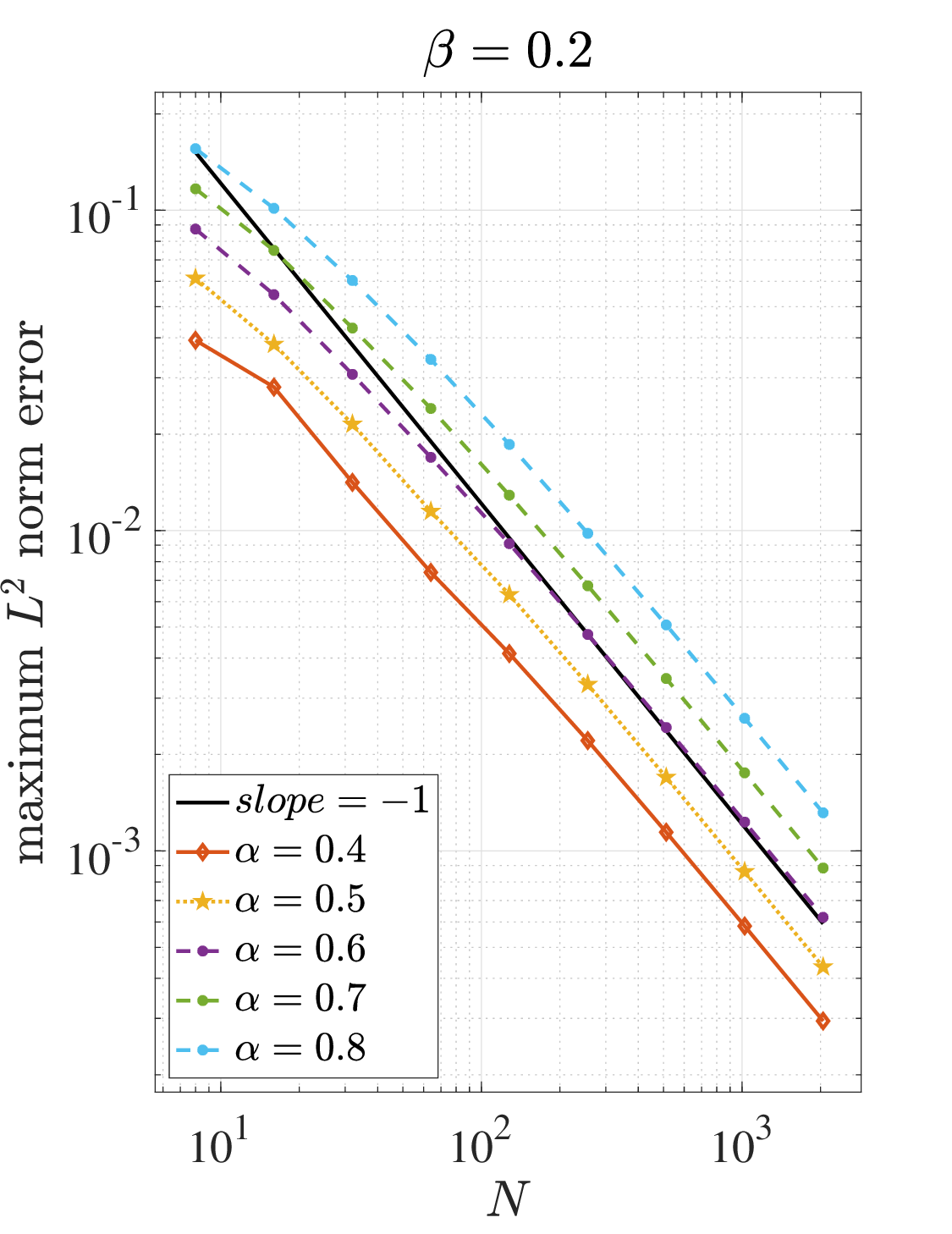}
	\includegraphics[height=6cm,width=6.8cm]{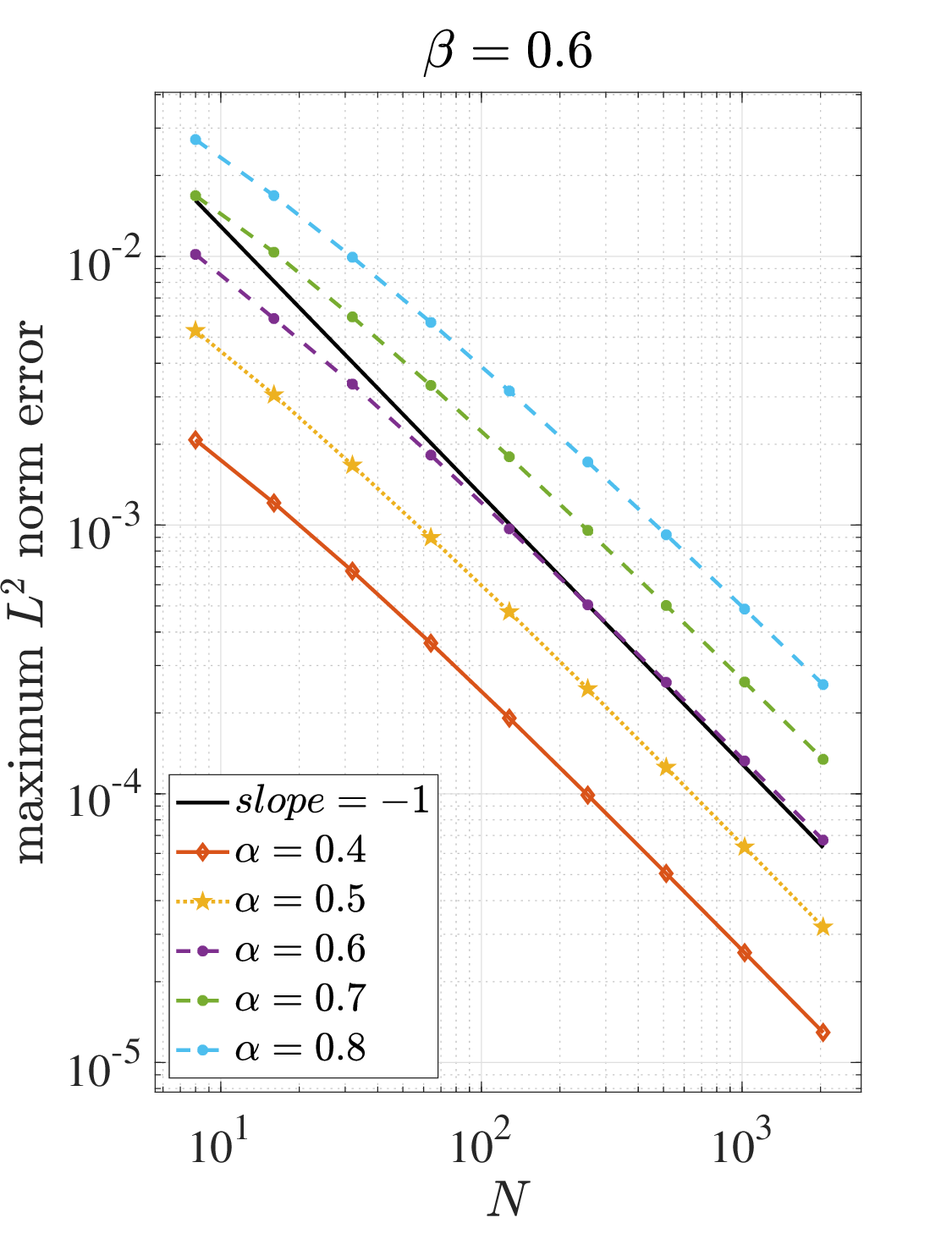}
	\caption{Maximum $L^2$ norm error for Example \ref{ex_PDEsmooth} with $\gamma=\frac{1}{\beta}$, $\Delta_x=\frac{1}{3000}$.}
	\label{fig:PDE_order}
\end{figure}

\begin{figure}[H]
	\centering
	\includegraphics[height=6cm,width=5.8cm]{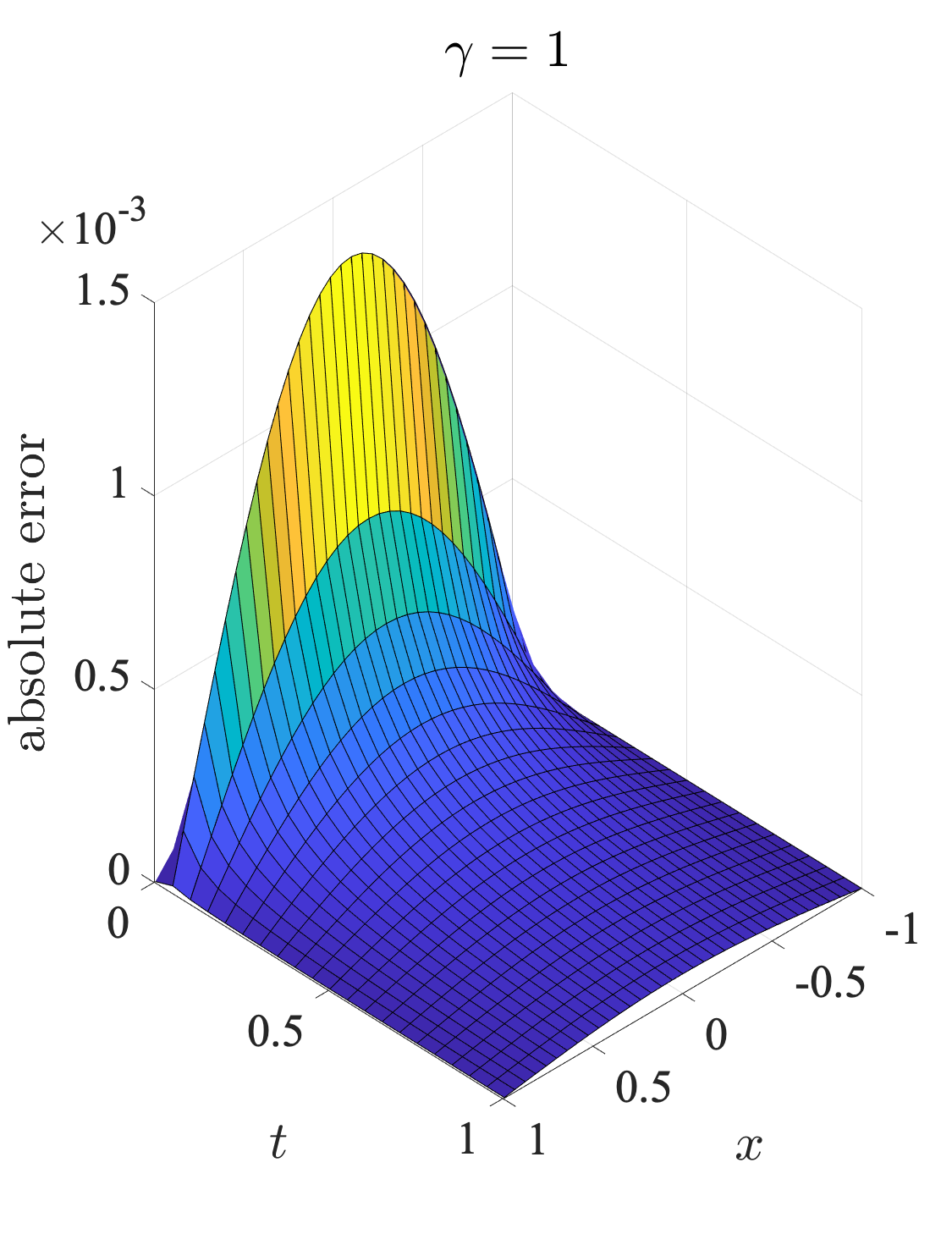}\quad
	\includegraphics[height=6cm,width=5.8cm]{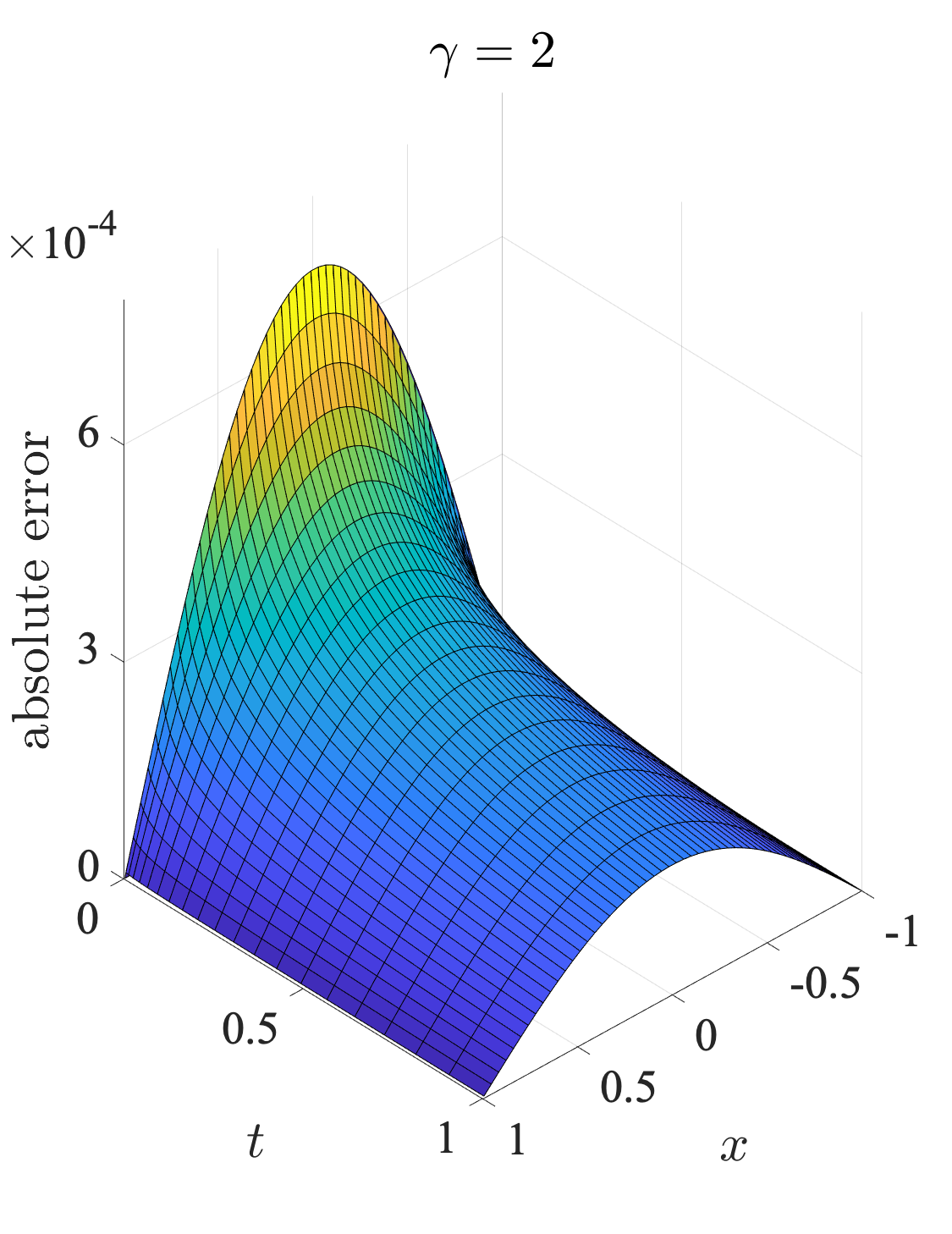}
	\caption{Absolute error  for Example \ref{ex_PDEsmooth} with $\alpha=\beta=0.5$, $\Delta_x=\frac{1}{512}$, $N=200$.}\label{fig:PDE_error}
\end{figure}

\begin{example}\label{ex_ODE_heavif}
	In this example, we consider
	\begin{equation*}
		D_t^\alpha u(t)+ u(t)=f(t),	\quad u(0)=1, \quad 0<t\leq
		1,
	\end{equation*}
	with an exact solution
	\begin{equation*}
		u(t)=1+t^{\beta_1}+H(t-r)(t-r)^{\beta_2},\quad t\ge0,\quad r\ge 0.
	\end{equation*}
 The corresponding source term can be given by
	\begin{equation*}
		f(t)=u(t)+\frac{\Gamma(\beta_1+1)}{\Gamma(\beta_1-\alpha+1)}t^{\beta_1-\alpha}+\frac{\Gamma(\beta_2+1)}{\Gamma(\beta_2-\alpha+1)}H(t-r)(t-r)^{\beta_2-\alpha},
	\end{equation*}
which is piecewise smooth.
\end{example}
For this example, we employ a graded mesh with higher density near the singularities $t=0$ and $t=r$, which is defined as follows
\begin{align}
	\nonumber
	N_1=\lfloor N r/T\rfloor,\quad &\lv(n)=n^{\gamma_1-1}(N_1-n+1)^{\gamma_2-1}N_1^{-\gamma_1-\gamma_2+1},\quad
	\tau(n)=\frac{r}{\Vert \lv\Vert _{l^1}}\lv, \quad 1\le n\le N_1,\\[.5em]
	\nonumber
	&t(0)=0,\quad t(n)=t(n-1)+\tau(n),\quad  1\le n\le N_1,\\[.5em]
	\nonumber
	&t(k)=t(N_1+1)+(T-r)(\frac{k-N_1}{N-N_1})^{\gamma_2}\quad N_1+1\le k\le N,\quad t(N)=T.
\end{align}
On the left-hand side of Figures~\ref{fig:ODEHeavifr028order} and \ref{fig:ODEHeavifr072order}, we depict the structure of the  mesh described above for different values of $r$, $\gamma_1$ and $\gamma_2$ with $N=1024$. The error plots on the right-hand side of Figures~\ref{fig:ODEHeavifr028order} and \ref{fig:ODEHeavifr072order} illustrate that the gCQ method is capable of handling problems with multiple singularities and demonstrates superior performance compared to the uniform method.
\begin{figure}[H]
	\centering
	\includegraphics[height=6cm,width=6.8cm]{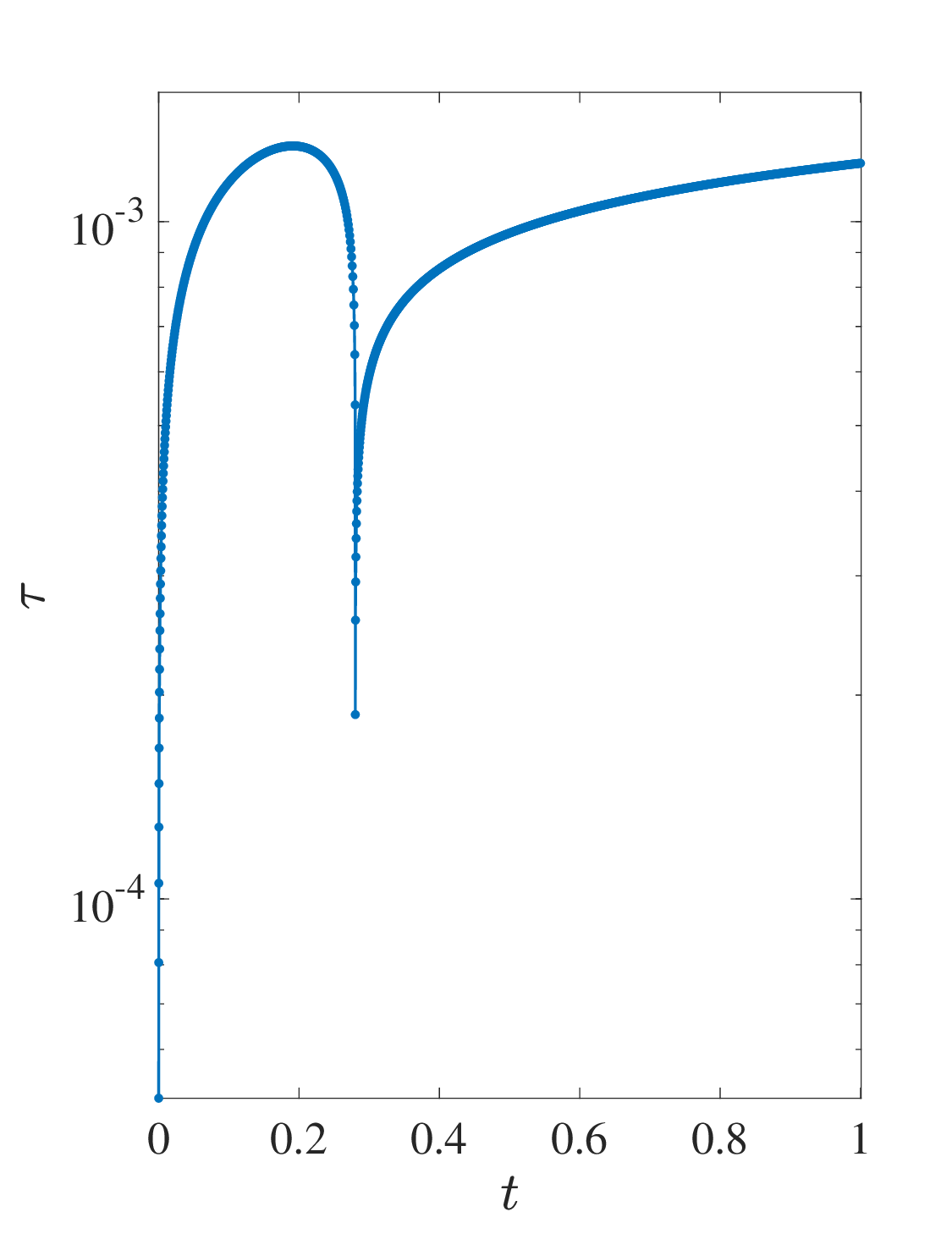}
	\includegraphics[height=6cm,width=6.8cm]{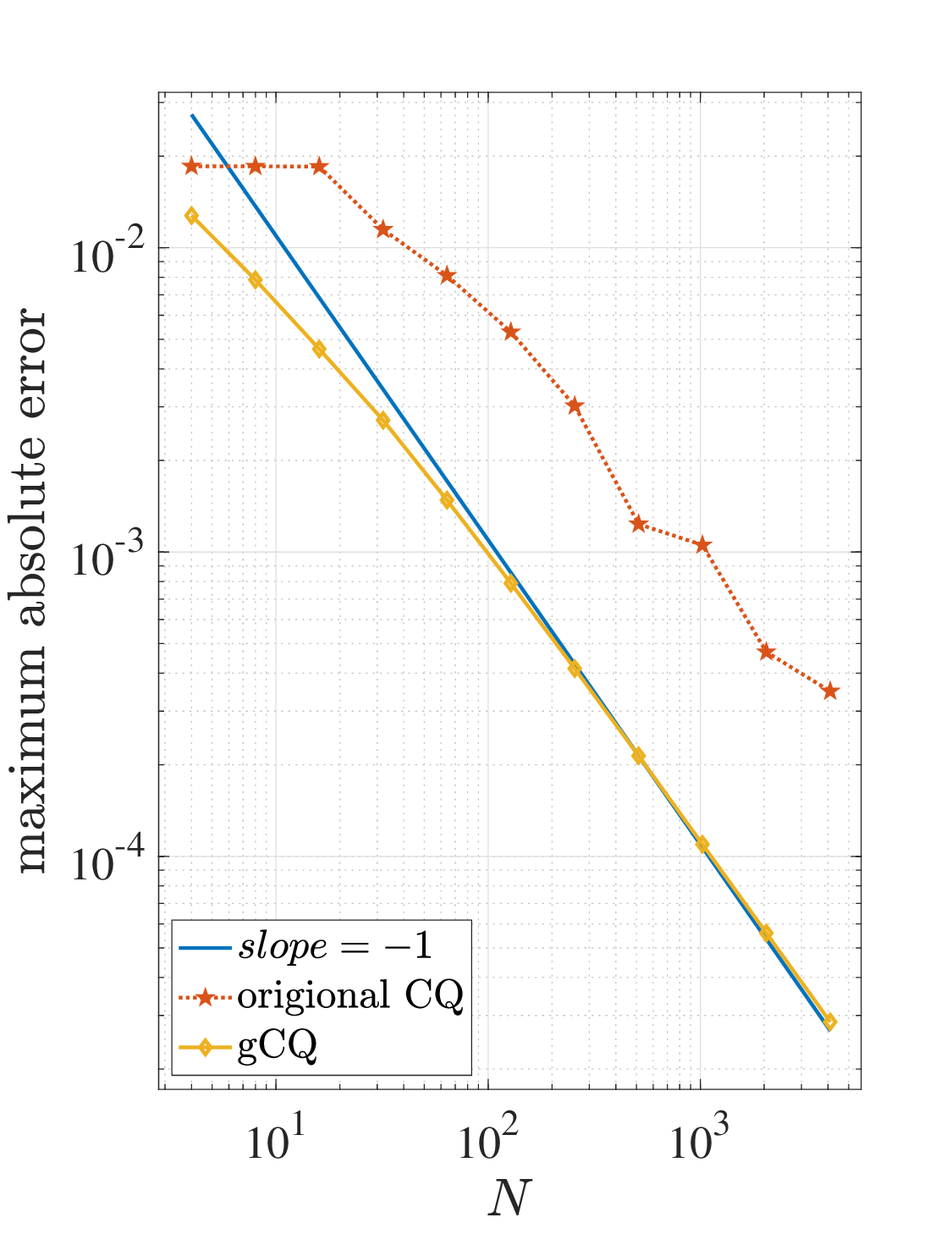}
\caption{Adaptive mesh used in the implementation of gCQ method with $N=1024$ (Left) and maximum absolute error (Right) for Example~\ref{ex_ODE_heavif} with $\alpha=0.4$, $\beta_1=0.6$, $\beta_2=0.8$,  $r=0.28$, $\gamma_1=1/\beta_1$, $\gamma_2=1/\beta_2$.} \label{fig:ODEHeavifr028order}
\end{figure}
\begin{figure}[H]
	\centering
	\includegraphics[height=6cm,width=6.8cm]{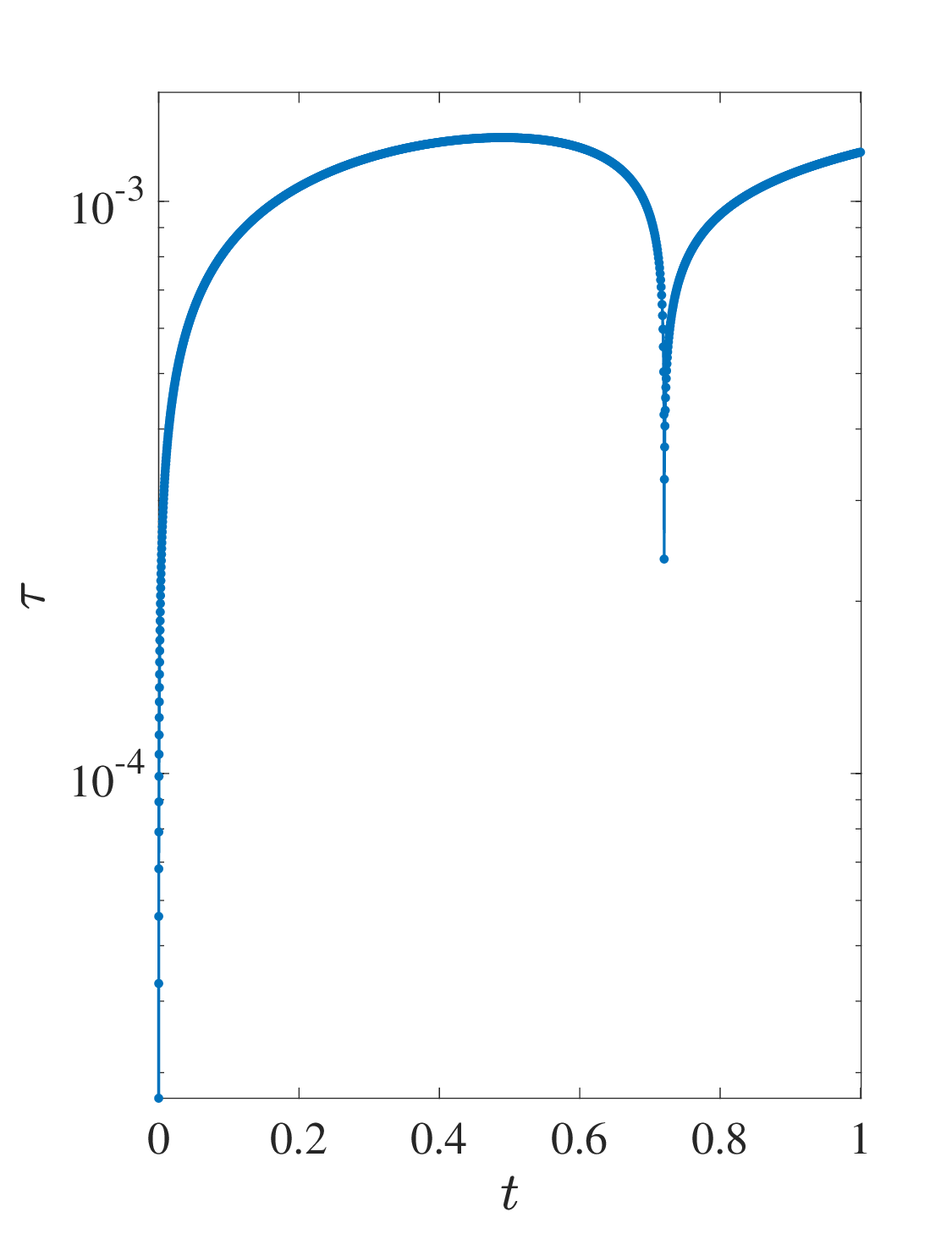}
	\includegraphics[height=6cm,width=6.8cm]{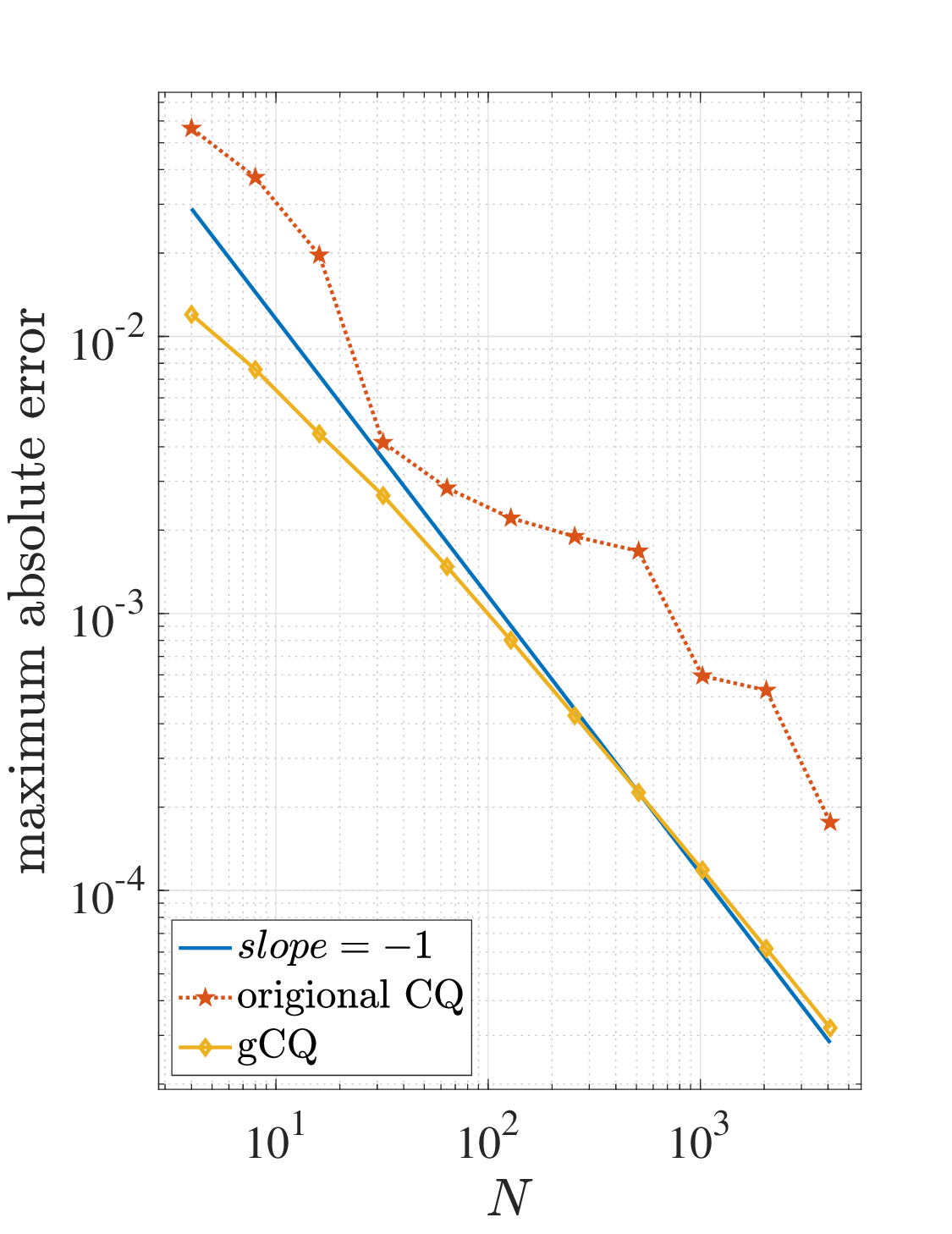}
	\caption{Adaptive mesh used in the implementation of gCQ method with $N=1024$ (Left) and maximum absolute error (Right) for Example~\ref{ex_ODE_heavif} with $\alpha=0.4$, $\beta_1=0.6$, $\beta_2=0.8$,  $r=0.72$, $\gamma_1=1/\beta_1$, $\gamma_2=1/\beta_2$.} \label{fig:ODEHeavifr072order}
\end{figure}

\begin{example}\label{ex_PDE_u01} In this example, we consider the same equation as in Example~\ref{ex_PDEsmooth}, but with $u_0(x)\notin D(A)$, namely.
\begin{equation}
\left\{\begin{array}{lll}\label{PDE_nons}
D_t^\alpha u-\Delta u=f(x,t),&\quad (x,t)\in \Omega\times(0,1],\\[.5em]
	u(x,0)=1,&\quad x\in \overline\Omega,\\[.5em]
	u(x,t)=0,&\quad (x,t)\in \partial \Omega\times(0,1].
\end{array}\right.
\end{equation}
in $\Omega=(-1,1)$ and with
\begin{equation*}
f(x,t)=\left(\frac{\Gamma(\beta+1)}{\Gamma(\beta-\alpha+1)}t^{\beta-\alpha}+\pi^2(1+t^{\beta})/4\right)\cos(\pi x/2).
\end{equation*}
\end{example}
As before we apply the linear finite element for the discretization in space of the problem. Then we approximate in time \eqref{fracdiffusion} with $A$ the resulting discrete Laplacian.

We apply the inversion method of the Laplace transform in \cite{LoPaScha} for the approximation of the term $E(t) u_0$ in  \eqref{fracdiff_exasol}. This gives
\begin{equation}\label{ILT_qua}
\mathcal{L}^{-1}[z^{\alpha-1}(z^\alpha I+A)^{-1}u_0](t_n) \approx\sum\limits_{l=-J}^{J}\omega_l(t_n)\e^{t_n z_l(t_n)}z^{\alpha-1}_l(t_n)(z^{\alpha}_l(t_n)I+A)^{-1}u_0,
\end{equation}
where
\[ \omega_l(t_n)=-\frac{h}{2\pi i}\varphi'(lh),\quad z_l(t_n)=\varphi(lh), \]
with
\begin{align*}
	&\varphi(s)=\mu(t_n)(1-\sin(\sigma+is)),\quad h=\frac{a(\theta)}{J},\quad\mu(t_n)=\frac{2\pi d J(1-\theta)}{ t_n a(\theta)},\\
		&a(\theta)={\rm{arccosh}}\left(\frac{1}{(1-\theta)\sin(\sigma)}\right),\quad \theta=1-\frac{1}{J},\quad d\in(0,\frac{\pi}{2}-\sigma),
\end{align*}
for $1\le n\le N$, with $\sigma=\frac{\pi}{4}$, $d=\frac{\pi}{6}$ and $J=50$. Notice that we use a high number of quadrature points per time point, which is certainly not necessary, since the method in \cite{LoPaScha} is able to approximate the inverse Laplace transform with the same quadrature uniformly on time windows of the form $[t_0, \Lambda t_0]$, with $\Lambda \gg 1$. In this way we are approximating the term $E(t)u_0$ with machine precision and observe only the error induced by the gCQ approximation of the convolution term in \eqref{fracdiff_exasol}.

The error in this example is measured with respect to a reference solution computed with double the time points, this is
\[
\max_{1<n<N}\Vert u_{\Delta_x}^n-\widetilde {u}_{\Delta_x}^{2n}\Vert_{L^2},\]
where  $\widetilde{u}^{2n}_{\Delta_x}$ is the  corresponding numerical solution on the finer time mesh
 \[ t_n=(n/(2N))^\gamma,  \qquad 0\leq n\leq 2N.\]
The optimal grading parameter  in \eqref{gmesh}  for this example is $\gamma=\max\left(1/\alpha,1/\beta\right)$.  And  the  values of $N_Q^{his}$ used in the quadrature formula \eqref{Ihis_quad} on different graded meshes \eqref{gmesh} are listed in Table~\ref{tab:PDEu01_Quad}.
\begin{table}[H]
\renewcommand{\captionfont}{\small}
	\centering
	\small
	\begin{spacing}{0.8}
		\caption{Value of $N_Q^{his}$  for Example~\ref{ex_PDE_u01} with  $\alpha=0.6$,    $\tol=10^{-8}$.}
		\label{tab:PDEu01_Quad}
	\end{spacing}
\begin{tabular}{p{1cm}<{\centering}p{1cm}<{\centering}p{1cm}<{\centering}p{1cm}<{\centering}p{1cm}<{\centering}p{1cm}<{\centering}p{1cm}<{\centering}p{1cm}<{\centering}<{\centering}p{1cm}<{\centering}p{1cm}<{\centering}}
\hline \specialrule{0pt}{2pt}{2pt}
$\gamma$\textbackslash$N$ & $16$ & $32$ & $64$ & $
128$ & $256$ & $512$ & $1024$ & $2048$ & $4096$ \\ \specialrule{0pt}{1.5pt}{1.5pt}\hline		\specialrule{0pt}{2pt}{2pt}
2 & 33 & 51 & 64 & 74 & 86 & 94 & 101 & 110 & 118 \\
 \specialrule{0pt}{1.5pt}{1.5pt}
4 & 64 & 109 & 145 & 173 & 200 & 228 & 254 & 280 & 306 \\
 \specialrule{0pt}{1.5pt}{1.5pt}
6 & 97 & 177 & 237 & 292 & 344 & 396 & 451 & 505 & 561 \\
 \specialrule{0pt}{1.5pt}{1.5pt}
8 & 128 & 252 & 345 & 433 & 519 & 604 & 696 & 784 & 879 \\
 \specialrule{0pt}{1.5pt}{1.5pt}
10 & 161 & 334 & 469 & 593 & 717 & 845 & 980 & 1119 & 1260 \\
\hline   \specialrule{0pt}{2pt}{2pt}
\end{tabular}
\end{table}
\begin{figure}[H]
\centering
\includegraphics[height=6cm,width=6.8cm]{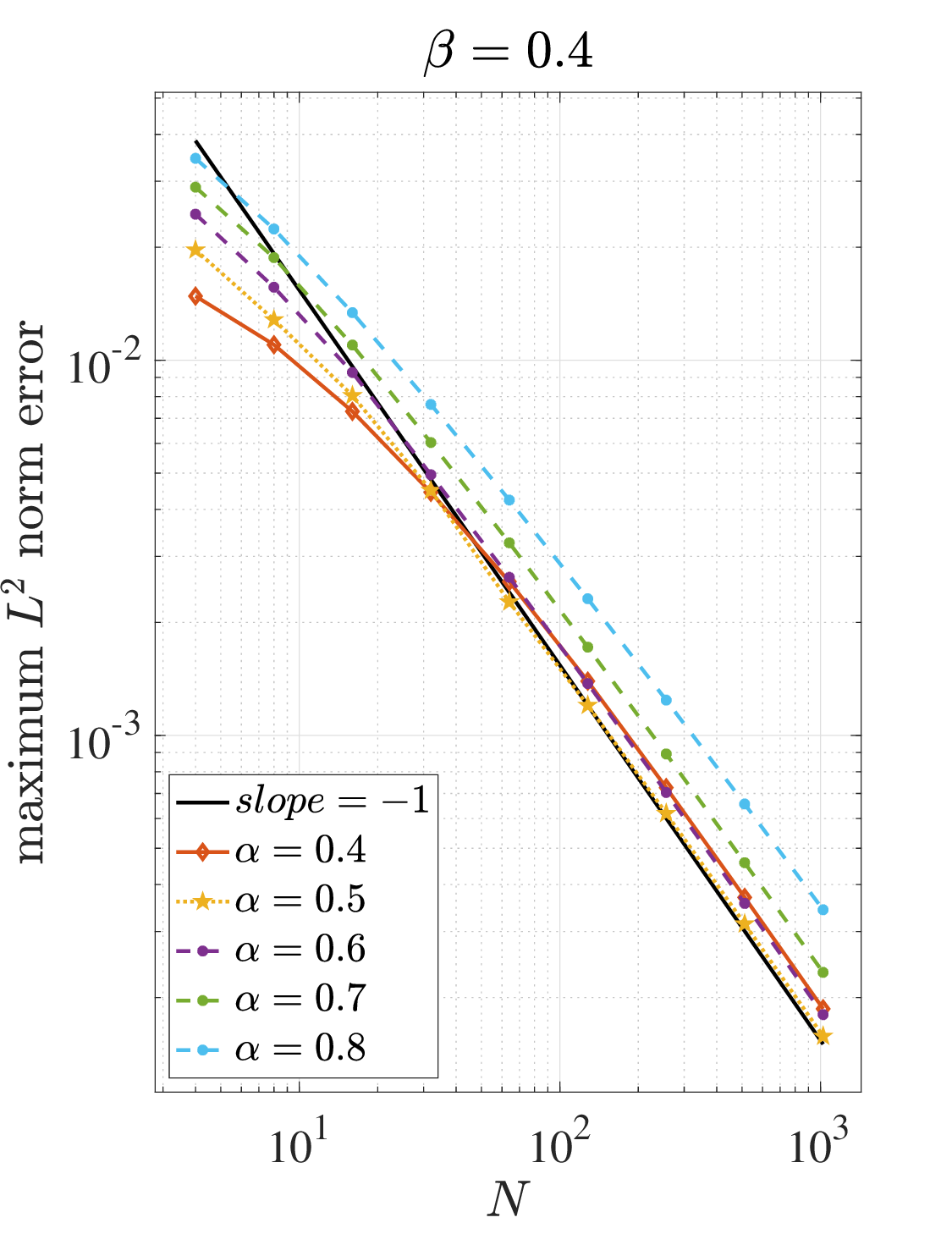}
\includegraphics[height=6cm,width=6.8cm]{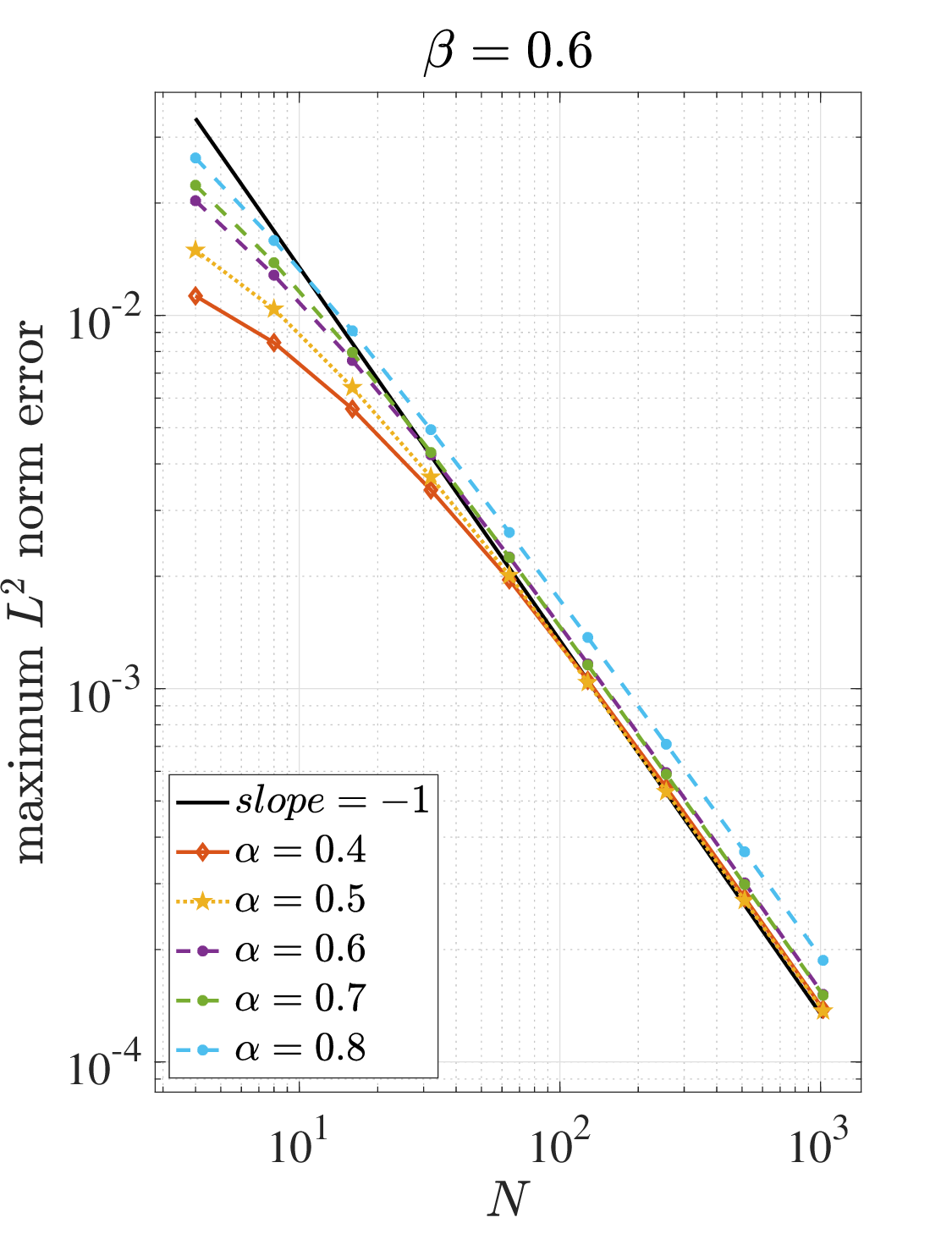}
 \caption{Maximum $L^2$ norm error  for Example \ref{ex_PDE_u01} with $\gamma=\max\left(\frac{1}{\alpha},\frac{1}{\beta}\right)$, $\Delta_x=\frac{1}{3000}$.}
 \label{fig:PDEu01_order}
    \end{figure}

\begin{figure}[H]
\centering
 \includegraphics[height=6cm,width=6.8cm]{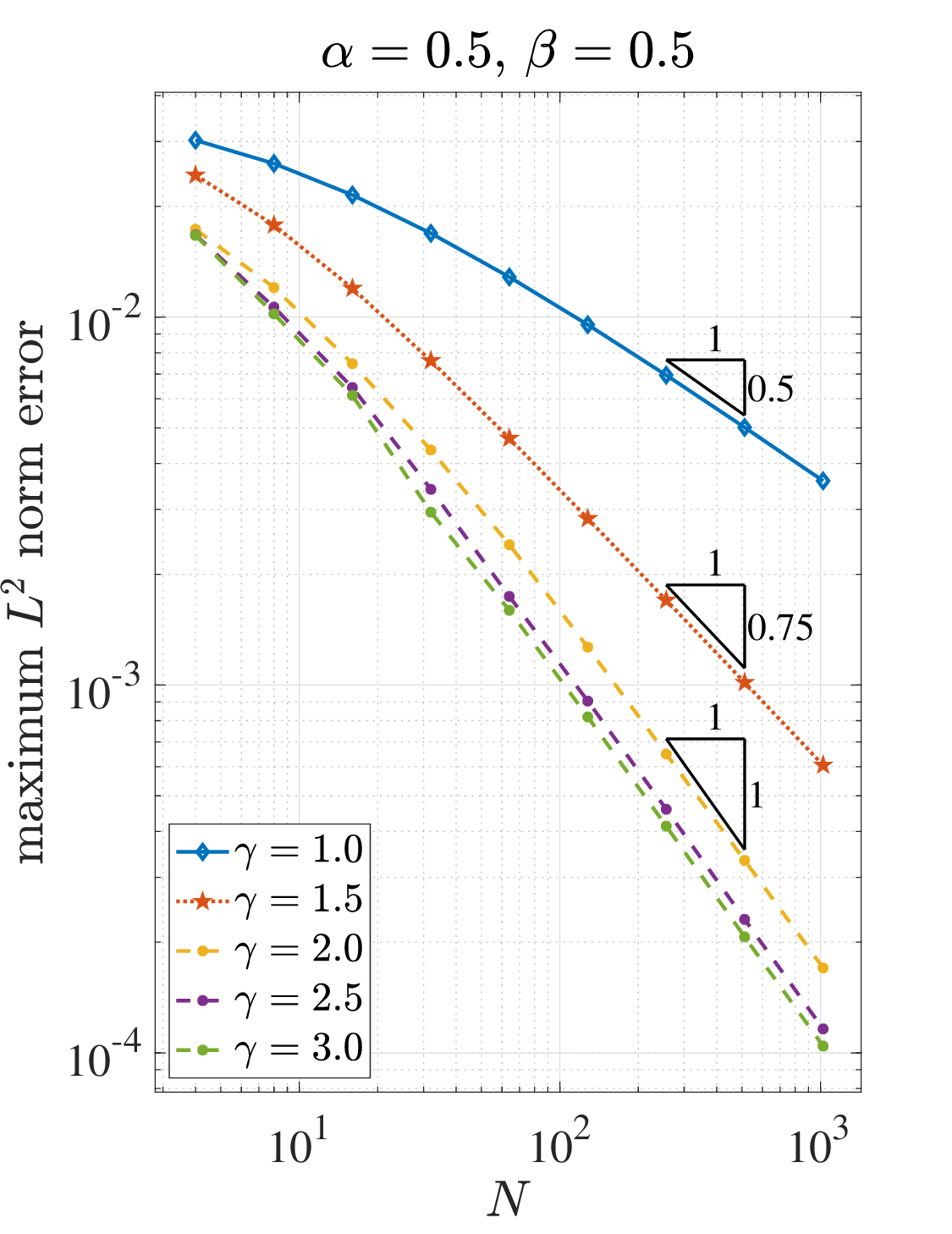}
 \includegraphics[height=6cm,width=6.8cm]{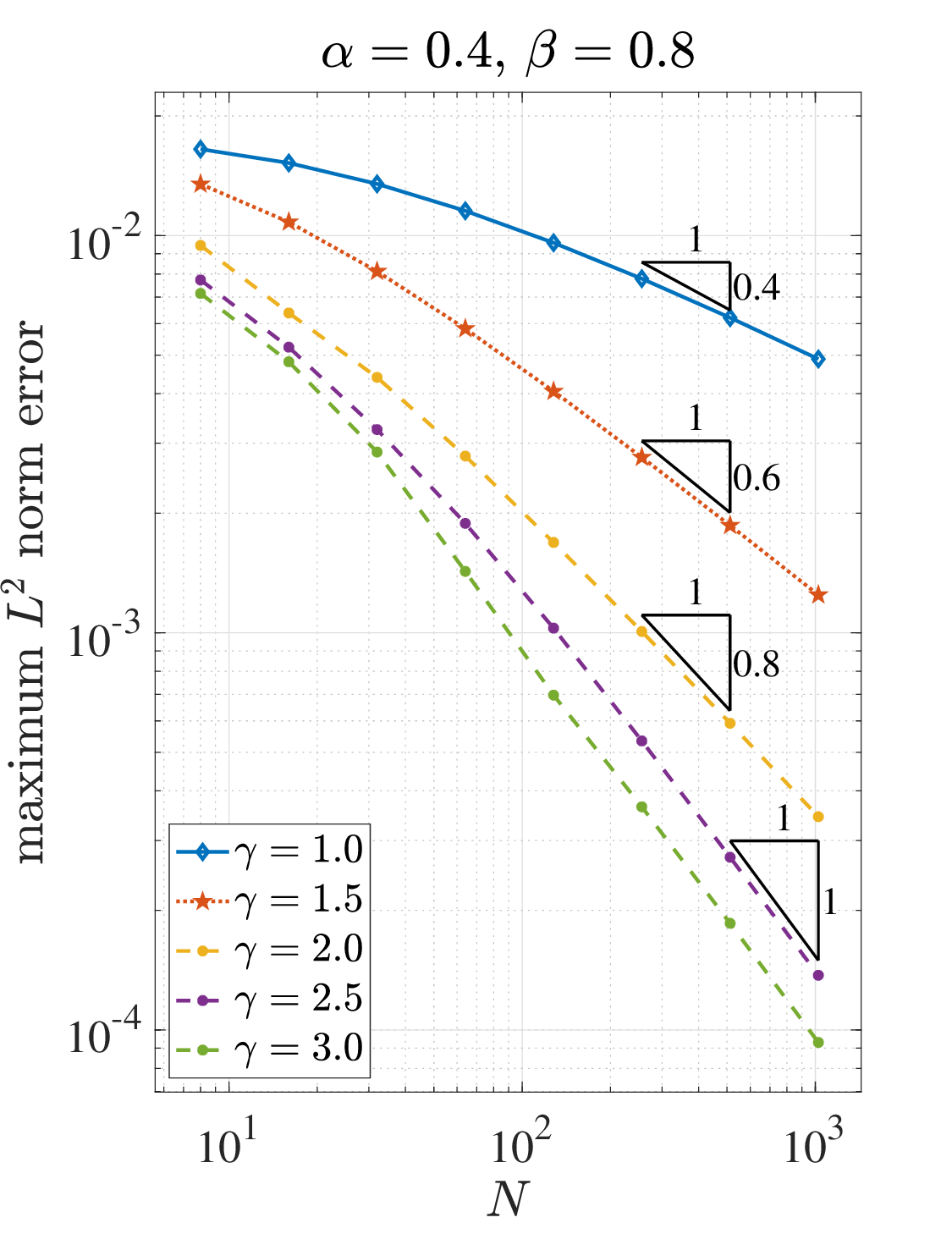}
    \caption{Maximum $L^2$ norm error  for Example \ref{ex_PDE_u01} with  $\Delta_x=\frac{1}{3000}$.}
 \label{fig:PDEu01_order_gam}
    \end{figure}

The results in Figure~\ref{fig:PDEu01_order} confirm that FEM-gCQ scheme for \eqref{PDE_nons} is convergent with full order one on the graded mesh \eqref{gmesh}, with   $\gamma=\max\left(1/\alpha,1/\beta\right)$. Figure~\ref{fig:PDEu01_order_gam} shows for different values of  $\gamma$ the convergence order of this method, which is equal to $\min\{1, \gamma \alpha,\gamma \beta\}$.

\section{Acknowledgements}
The authors acknowledge very useful discussions with Lehel Banjai, whose valuable comments have helped significantly to improve the presentation of this paper. They also acknowledge the constructive and useful comments of the anonymous referees.

The research of the first author has been mostly developed during a stay at University of Malaga supported by  the China Scholarship Council (CSC) (No. 202106720028). The second author has been supported by the ``Beca Leonardo for Researchers and Cultural Creators 2022'' granted by the BBVA Foundation and by the ``Proyecto 16 - proyecto G Plan Propio" of the University of Malaga. The BBVA Foundation takes no responsibility for the opinions, statements and contents of this project, which are entirely the responsibility of its authors.

\section{Data availability}
The codes implementing the examples discussed in this article are available upon request to the authors, exclusively for academic use. No other data are associated to the manuscript.

\appendix

\section{Continuity of the fractional derivative}\label{sec:regularity}
We adopt here the notation in \cite{SamkoKilbasMarichev}.
Let us denote $H^{\lambda}([0,T])$ the space of H\"older continuous functions $f(t)$ of order $0 < \lambda \le 1$, this is
\[
|f(t_1)-f(t_2)| \le A |t_1-t_2|^{\lambda},
\]
for some constant $A$. Let $\rho(x)$ be a non negative function. We denote $H^{\lambda}(\rho)= H^{(0,T);\lambda}(\rho)$ the space of functions $f$ such that $\rho f \in H^{\lambda}$, this is, $f \in H^{\lambda}(\rho)$ if
\[
f(t) = \frac{f_0(t)}{\rho(t)},  \mbox{ with } \ f_0(t) \in H^{\lambda}.
\]
Assume now that
\[
\rho(t) = \prod_{k=1}^{K} (t-t_k)^{\mu_k},
\]
for some $\mu_k \in \bR$ and $t_k\in [0,T]$. Then we denote $H^{\lambda}_0(\rho)=H^{\lambda}_0([0,T];\rho)$ the space of functions $f \in H^{\lambda}([0,T];\rho)$ such that $f_0(t_k)=0$.

For $\lambda=m+\sigma$ with $m=0,1,2,\dots$ and $0< \sigma \le 1$, we say that $f\in H^{\lambda}([0,T])$ if $f\in C^{m}([0,T])$ and $f^{(m)}\in H^{\sigma}([0,T])$. It is clear that $C^{m+1}([0,T]) \subset H^{\lambda}([0,T])$.

Assume now $f(t)=t^{\beta} g(t)$ with $g\in C^{m+1}([0,T])$, $m\ge 0$, and $\beta >-1$, $\beta \notin \bZ$. By using the Taylor expansion of $g$ we can write
\[
g(t) = \sum_{\ell=0}^{m} \frac{ g^{(\ell)}(0)}{\ell!} t^{\ell} + \bar{g}(t),
\]
with the remainder $\bar{g} \in C^{m+1}([0,T])$, satisfying $\bar{g}^{(\ell)}(0)=0$, for $\ell=0,\dots,m$. Then
\[
f(t) = \sum_{\ell=0}^{m} \frac{g^{(\ell)}(0)}{\ell!} t^{\beta+\ell} + t^{\beta} \bar{g}(t).
\]
Now, for  $\nu=\lceil \beta+p \rceil$, with $p\in \bN$, we compute
\[
f^{(\beta+p)}(t) =  \sum_{\ell=0}^{m} \frac{g^{(\ell)}(0)}{\ell!} \frac{\Gamma(\beta+\ell+1)}{\Gamma(\nu-p+\ell+1)} \frac{d^{\nu}}{dt^{\nu}} t^{\ell+\nu-p} + \frac{d^{\nu}}{dt^{\nu}} \mathcal{I}^{\nu-(\beta+p)}[t^{\beta} \bar{g}](t).
\]
Since the integer exponents $\ell+\nu-p \ge 0$, the addends in the summation above are either $0$ or continuous in $[0,T]$. Let us now consider the fractional integral of order $\alpha \in(0,1)$ of the function $t^{\beta}\bar{g}$, which is given by
\[
\partial_t^{-\alpha}\left[ t^{\beta} \bar{g}(t) \right] = \frac{1}{\Gamma(\alpha)}\int_0^t (t-s)^{\alpha-1} s^{\beta}\bar{g}(s)\,ds.
\]
Integration by parts and the vanishing moments of $\bar{g}$ at the origin imply that
\[
\partial_t^{-\alpha}\left( t^{\beta} \bar{g}(t) \right) =\frac{1}{\Gamma(\alpha+m+1)}\int_0^t (t-s)^{\alpha+m} \partial_s^{m+1}\left[s^{\beta}\bar{g}(s)\right]\,ds.
\]
We can differentiate up to $m+1$ times with respect to $t$ in the expression above, while keeping a convergent integral. Differentiation $m+1$ times gives
\[
\partial_t^{m+1}\partial_t^{-\alpha}\left( t^{\beta} \bar{g}(t) \right) = \frac{1}{\Gamma(\alpha)}\int_0^t (t-s)^{\alpha-1} \partial_s^{m+1}\left[s^{\beta}\bar{g}(s)\right]\,ds = \partial_t^{-\alpha}\partial_t^{m+1}\left( t^{\beta} \bar{g}(t) \right).
\]
The fractional derivative above is H\"older continuous if $\bar{g}^{(m+1)}$ is H\"olderian, by \cite[Lemma 13.2]{SamkoKilbasMarichev}, this is $\bar{g} \in H^{m+1+\sigma}$ for some $\sigma \in (0,1)$ and, in particular, if $g\in C^{m+2}([0,T])$.

\bibliographystyle{abbrv}
\bibliography{bibmlf.bib}

\end{document}